\documentclass[11pt]{article}
\usepackage[top=2cm,bottom=2.5cm,right=2.5cm,left=2.5cm]{geometry}
\usepackage[english]{babel}
\usepackage[T1]{fontenc}
\usepackage{times}
\usepackage{mathtools, bm}
\usepackage{amssymb, bm}
\usepackage{graphicx}
\usepackage{mathrsfs}
\usepackage{stmaryrd}
\usepackage{amsthm}
\usepackage{listings}
\usepackage{pict2e}
\usepackage{pgf, tikz}
\usepackage{dsfont}
\usepackage{algorithm2e}
\usepackage{algorithmic}
\usepackage{multicol}
\usepackage{hyperref}

\newcommand{\modif}[1]{{\color{black}#1}}

\thicklines

\newtheorem{theorem}{Theorem}
\newtheorem{proposition}{Assumption}
\newtheorem{corollary}{Corollary}
\newtheorem{lemma}{Lemma}
\newtheorem{remark}{Remark}
\newtheorem{definition}{Definition}
\newtheorem{example}{Example}

\numberwithin{equation}{section} 

\date{}
\author{Arnaud Guillin\footnote{Universit\'{e} Clermont Auvergne, CNRS, LMBP, F-63000 Clermont-Ferrand, France and Institut Universitaire de France. Email : arnaud.guillin@uca.fr},$\ $ Pierre Le Bris\footnote{Institut des Hautes Etudes Scientifiques, F-91440 Bures-sur-Yvette, France. Email: lebris@ihes.fr}$\ $ and Pierre Monmarch\'{e}\footnote{Sorbonne Universit\'{e}, CNRS, LJLL, F-75005, Paris, France. Email : pierre.monmarche@sorbonne-universite.fr}}
\title{Uniform in time propagation of chaos for the 2D vortex model and other singular stochastic systems}

\begin{document}
\maketitle

\begin{abstract}
In this article, we adapt the work of Jabin and Wang in \cite{JW18} to show the first result of uniform in time propagation of chaos for a class of singular interaction kernels. In particular, our models contain the Biot-Savart kernel on the torus and thus the 2D vortex model.
\end{abstract}



%
%
%
%

\section{Introduction}

\subsection{Framework}
Our main subject is the convergence of the law of a stochastic particles system with mean field singular interactions towards its non linear limit. More precisely we will establish the first quantitative bounds in the number of particles uniformly in time. Let $K:\mathds{T}^d\rightarrow \mathbb R^d$ be an \emph{interaction kernel} on the $d$-dimensional ($d\geqslant 2$)  $1$-periodic torus $\mathds{T}^d$ (represented as $[-\frac{1}{2},\frac{1}{2}]^d$), on which we will specify some assumptions later. In this paper, we consider the non linear stochastic differential equation of \textit{McKean-Vlasov type}
\begin{equation}\label{o_Lang}
\left\{
    \begin{array}{ll}
        dX_t=\sqrt{2}dB_t+K\ast\bar{\rho}_t(X_t)dt\\
        \bar{\rho}_t=\text{\modif{Density of }Law}(X_t),
    \end{array}
\right.
\end{equation}
where $X_t\in \mathds{T}^d$, $(B_t)_{t\geq0}$ is a $d$-dimensional Brownian motion  and    $f\ast g(x)=\int_{\mathds{T}^d} f(x-y)g(y)dy$ stands for the convolution operation on the  torus. The \modif{density} $\bar{\rho}_t$ satisfies
\begin{equation}\label{o_Liou}
\partial_t\bar{\rho}_t=-\nabla\cdot\left(\bar{\rho}_t\left(K\ast\bar{\rho}_t\right)\right)+\Delta\bar{\rho}_t.
\end{equation}
In the other words, the non-linear Equation~\eqref{o_Liou} has the following natural probabilistic interpretation :  the solution $\bar{\rho}_t$ is the density of the law at time $t$ of the $\mathds{T}^{d}$ valued process $(X_t)_{t\geq0}$ evolving according to \eqref{o_Lang}. As we understand \eqref{o_Lang} to be the motion of a particle interacting with its own law, \eqref{o_Liou} thus describes the dynamic of a cloud of charged particles (where $(X_t)_{t\geq0}$ would be one particle). In particular, it holds importance in plasma physics, see \cite{Vlasov_1968}. We also consider the associated system of particles, describing the motion of $N$ particles interacting  with one another through the interaction kernel $K$.
\begin{equation}\label{o_Lang_part}
dX^i_t=\sqrt{2}dB^i_t+\frac{1}{N}\sum_{j=1}^NK(X^i_t-X^j_t)dt,
\end{equation}
where $X^i_t\in\mathds{T}^d$ is the position at time $t$ of the i-th particle, and $(B^i_t, 1\leq i\leq N)$ are independent Brownian motions in $\mathds{T}^d$. We assume that $(X^i_0)_{i=1,..,N}$ are exchangeable, i.e. have a  law which is invariant by permutation of the particles, so that this property is true for all times.
We denote by $\rho_N$ the \modif{density of the} law of the system of particles, formally satisfying 
\begin{equation}\label{part_liou}
\partial_t\rho_N=-\sum_{i=1}^N\nabla_{x_i}\cdot\left(\left(\frac{1}{N}\sum_{j=1}^NK(x_i-x_j)\right)\rho_N\right)+\sum_{i=1}^N\Delta_{x_i}\rho_N.
\end{equation}
We define $\rho_N^k$ the \modif{density of the} law of the first $k$ marginals of the $N$ particles system
\begin{align*}
\rho_N^k(t,x_1,..,x_k)=\int_{\mathds{T}^{(N-k)d}}\rho_N(t,x_1,..,x_N)dx_{k+1}...dx_N,
\end{align*}
which is also, thanks to the exchangeability of particles, the \modif{density of the} law of any $k$ marginals. More precisely, in this work, we focus on the equation \eqref{part_liou} and we will not address the question of the well-posedness of the stochastic equation  \eqref{o_Lang_part}.

Here, although we will consider general assumptions on $K$, the main example motivating our work is the singular interaction kernel known as \textit{the Biot-Savart} kernel, defined in $\mathbb{R}^2$ by
\begin{equation}\label{BS}
K(x)=\frac{1}{2\pi}\frac{x^{\perp}}{|x|^2}=\frac{1}{2\pi}\left(-\frac{x_2}{|x|^2},\frac{x_1}{|x|^2}\right).
\end{equation}
Consider the 2D incompressible Navier-Stokes system on $\modif{x}\in\mathbb{R}^2$ 
\begin{align*}
\partial_t u=&-u\cdot\nabla u-\nabla p+\Delta u \\
\nabla\cdot u=&0,
\end{align*}
where $p$ is the local pressure. Taking the curl of the equation above, we get that $\omega(t,x)=\nabla\times u(t,x)$ satisfies \eqref{o_Liou} with $K$ given by \eqref{BS} (see for instance Chapter 1 of \cite{MP94}).

One can see equation~\eqref{o_Lang_part} as an approximation of equation~\eqref{o_Lang}, where the law $\bar{\rho}_t$ is replaced by the empirical measure $\frac{1}{N}\sum_{i=1}^N\delta_{X^i_t}$. It is well known, at least in a setting where the interaction kernel $K$ is Lipschitz continuous (\cite{Mel96}, \cite{Szn91}), that, under some mild conditions on $K$, for all fixed $k\in\mathbb N$ \modif{and all $t\geq0$}, $\rho_N^k\modif{(t,\cdot)}$ converges toward $\bar{\rho}_k\modif{(t,\cdot)}=\bar{\rho}_{\modif{t}}^{\otimes k}$ as $N$ tends to infinity, where $\bar{\rho}_{\modif{t}}$ is the \modif{density of the} law of $X_{\modif{t}}$ solution of \eqref{o_Lang}. Thus, provided the particles start independent, they will stay (more or less) independent, as the law of any $k$-uplet of particles converges toward a tensorized law. The expression \textit{propagation of chaos} to describe this behavior was coined by Kac \cite{Kac56}), and we refer to Sznitman \cite{Szn91} for a landmark study of the phenomenon. Of course there is a huge literature on propagation of chaos however limited for uniform in time results, and always when the interaction potential is regular, see Malrieu \cite{Mal01} for an example by a coupling approach under convexity conditions and the recent Durmus $\& al$ \cite{DEGZ20} via reflection coupling allowing non convexity but where the interaction is considered small and acts mainly as a perturbation. For more recent results we refer to \cite{Lac21} \modif{(and its uniform in time extension in \cite{LLF22})} for a nice new approach for propagation of chaos furnishing better speed but strong assumptions on the interactions (regularity, integrability), including a nice survey of the existing results, and \cite{DT21} using Lions derivatives for uniform in time results on the torus but also under regularity assumptions on the interaction kernel.

Hence, both these classical and recent results do not apply to the Biot-Savart kernel, which is singular at $0$. For a convergence without rate, and specific to the vortex 2D equation, a first striking result appeared in \cite{FHM14}, relying on proving that close encounters of particles are rare and that the possible limits of the particles system are made of solutions of the nonlinear SDE. As a second step, in the recent work \cite{JW18}, Jabin and Wang have proven that propagation of chaos still holds in this case with a {\it quantitative} rate. The goal of the present paper is to extend their works and show a {\it quantitative propagation of chaos uniform in time}. We refer to \cite{FHM14,JW18,BJW19,BJW19b,BJW20} for detailed discussions on the literature concerning propagation of chaos with singular kernels, which is still at its beginning for quantitative rates. \modif{Shortly after this work was submitted, an alternative approach to global in time estimates was developed in \cite{RS23}, see also the very recent preprint \cite{CRS23}.}

\modif{Obtaining uniform in time estimates for propagation of chaos is an important challenge to tackle. One of its applications for instance concerns the use of particle system, which can easily be simulated numerically, to approximate the solution of a nonlinear physics motivated problem, such as here the vorticity equation arising from fluid mechanics. Likewise, it provides a framework for studying noisy gradient descent used in Machine Learning (see the recent \cite{CRW22}) and thus attracts some attention.}

The approach of Jabin and Wang \cite{JW18} is to compute the time evolution of the relative entropy of $\rho_N$ with respect to $\bar \rho_N$ and then  to use an integration by parts to deal with the singularity of $K$ thanks to the regularity of the probability density $\bar{\rho}_t$. In order to improve  this argument to get uniform in time propagation of chaos, our main contribution is the proof of time-uniform bounds for \modif{$\bar{\rho}_t$}, \modif{ in Lemma~\ref{Bornes_derivees},} from which a time-uniform logarithmic Sobolev inequality is deduced. From the latter, in the spirit of the work of Malrieu \cite{Mal01} in the smooth and convex case, the Fisher information appearing in the entropy dissipation yields a control on the relative entropy itself, inducing the time uniformity. {\modif However a major difficulty is that this quantities are expressed in terms of the solution of the nonlinear equation. We then have to prove a logarithmic Sobolev inequality, uniformly in time, for $\bar\rho_t$, and a sufficient decay of the derivatives of $\bar\rho_t$. To do so, it requires new estimates on regularity and a priori bounds of the solutions of non linear 2D vortex equation. Indeed, we prove that the bounds on the derivative of $\bar{\rho}_t$ decay sufficiently fast (see again Lemma~\ref{Bornes_derivees}) to ensure uniform in time convergence without smallness assumption on the interaction.} Finally, the remaining error term in the entropy evolution due to the difference between \eqref{o_Liou} and \eqref{part_liou} is tackled thanks to a law of large number already used in \cite{JW18}. \modif{Compared to \cite{FHM14} we thus obtain a quantitative and uniform in time result.}

The organization of the article is as follows. For the remaining of this section, we state the main theorem as well as the various assumptions on both the initial condition and the interaction kernel $K$. In Section~\ref{Preliminary_work} we gather various tools that will be useful later on: we state the regularity of the solutions, the existence of uniform in time bounds on the density and its derivatives, and we prove a logarithmic Sobolev inequality. Finally, in Section~\ref{Section_Proof}, we prove the uniform in time propagation of chaos following the method described in \cite{JW18}.

\subsection{Main results}

First, let us describe the assumptions made on the initial condition. Unless otherwise specified,   $L^p$ and $H^p$ respectively refer to the spaces $L^p(\mathds{T}^d)$ and $H^p(\mathds{T}^d)$. Given $\lambda>1$, we denote by $\mathcal{C}^{\infty}_\lambda(\mathcal{X})$ the set of functions $f$ in $\mathcal{C}^\infty(\mathcal{X})$ such that  $0<\frac{1}{\lambda}\leq f\leq \lambda<\infty$, and $\mathcal C^\infty_{>0}(\mathcal X)=\cup_{\lambda>1}\mathcal{C}^{\infty}_\lambda(\mathcal{X})$, which is simply the set of positive smooth functions when $\mathcal X$ is compact.

\begin{proposition}\label{assumption_rho_0}
We make the following assumptions on $\bar{\rho}_0$ :
\begin{description}
\item[$\bullet$]There is $\lambda>1$ such that $\bar{\rho}_0\in\mathcal{C}^{\infty}_\lambda(\mathds{T}^d)$
\item[$\bullet$] For all $n\geq 1$, $C^0_n:=\|\nabla^n\bar{\rho}_0\|_{L^\infty}<\infty $
\end{description}
\end{proposition}

\modif{
\begin{remark}
Let us discuss the smoothness assumption on the initial condition. Via Theorem~\ref{exist} below, which follows from the result of \cite{Ben94}, this will ensure the smoothness of $\bar{\rho}_t$. This fact (and the fact that we consider, as we will see later, a smooth solution $\rho_N$ of \eqref{part_liou}) allows us to justify all calculations in comfortable way. This could however be improved. First, as in \cite{JW18}, the calculations should hold for any entropy solution of \eqref{part_liou}. Second, it is also shown in \cite{Ben94}, in the case of the vorticity equation, that an initial condition in $L^1$ yields existence and uniqueness of a solution of \eqref{o_Liou} which is smooth for positive times. One could thus think of using the non-uniform in time result of \cite{JW18} on a small time interval $[0,\epsilon]$, and then complete the proof on $[\epsilon,\infty[$ with our result. We would then require some bounds on $\bar{\rho}_\epsilon$ and its derivatives of a sufficient order (depending on the Sobolev embedding, see the proof of Lemma~\ref{Bornes_derivees} below) that we could propagate in time.

For the sake of clarity and conciseness however we choose not to insist in this direction.
\end{remark}
}

Let us describe the assumptions on the interaction kernel $K$. Below, $\nabla \cdot$ stands for the divergence operator.

\begin{proposition}\label{assumption_K}
We make the following assumptions on $K$ :
\begin{description}
\item[$\bullet$] $\|K\|_{L^1}<\infty$.
\item[$\bullet$] In the sense of distributions, $\nabla\cdot K=0$,
\item[$\bullet$] There is a matrix field $V\in L^{\infty}$ such that $K=\nabla\cdot V$, i.e for $1\leq\alpha\leq d$, $K_\alpha=\sum_{\beta=1}^d\partial_\beta V_{\alpha,\beta}$. 
\end{description}
\end{proposition}

The problem of finding a matrix field $V\in L^{\infty}(\mathds{T}^d)$ such that $K=\nabla\cdot V$ for a given $K$ is a complex mathematical question. We refer to \cite{BB03} and \cite{PT08} and the references therein for a more detailed discussion on the literature. As it was noted in Proposition~2 of \cite{JW18}, the existence of such a matrix $V$ is true for any kernel $K\in L^d$ (using the results of \cite{BB03}), or  for any kernel $K$ such that $\exists M>0, \forall x\in\mathds{T}^d, |K(x)|\leq M/|x|$ (using the results of \cite{PT08}).

\begin{remark}
If a function $a$ satisfies $\nabla\cdot a=0$, then for $\psi :\mathds{T}^d\mapsto \mathbb{R}$ we have $\nabla\cdot (a\psi)=(a\cdot\nabla)\psi$
\end{remark}

Suppose $\tilde K$ is an interaction kernel in $\mathbb{R}^d$ (such as the Biot-Savart kernel). It is possible to periodize $\tilde K$ on the torus as follows. For $f$ a function on the torus (identified as a  $1$-periodic function on $\mathbb R^d$), writing $f\ast_{\mathcal{X}}g(x)=\int_{\mathcal{X}} f(x-y)g(y)dy$ the convolution operator on a  space $\mathcal X$, 
\begin{align*}
\tilde K\ast_{\mathbb{R}^d} f(x)=\int_{\mathbb{R}^d}\tilde K(x-y)f(y)dy&=\sum_{k\in\mathbb{Z}^d}\int_{\mathds{T}^d} \tilde K(x-y+k)f(y-k)dy\\
&=\int_{\mathds{T}^d} \left(\sum_{k\in\mathbb{Z}^d}\tilde K(x-y+k)\right)f(y)dy,
\end{align*}
and thus $\tilde K\ast_{\mathbb{R}^d} f(x)=K\ast_{\mathds{T}^d} f(x)$, where $K(x)=\sum_{k\in\mathbb{Z}^d}\tilde K(x+k)$. In particular, the periodized Biot-Savart kernel obtained by taking $\tilde K$ from \eqref{BS} reads
\begin{equation}\label{BS_tore}
K(x)=\frac{1}{2\pi}\frac{x^{\perp}}{|x|^2}+\frac{1}{2\pi}\sum_{k\in\mathbb{Z}^2, k\neq 0}\frac{(x-k)^\perp}{|x-k|^2}:=\tilde{K}(x)+K_0(x).
\end{equation}
It has been shown that the sum defining $K_0$ converges (in the sense that $K_0(x)=\lim_{N\rightarrow\infty}\sum_{|k|^2\leq N, k\neq0}\frac{(x-k)^\perp}{|x-k|^2}$) in $\mathcal{C}^\infty$ (see for instance \cite{Sch96}). It is straightforward to check that $K$ is periodic, bounded in $L^1$, and divergence free. Finally, Proposition~2 of \cite{JW18} yields the existence of $V\in L^\infty$ such that $K=\nabla\cdot V$. As a consequence, Assumption~\ref{assumption_K} holds in the case of the periodized Biot-Savart kernel.

\begin{remark}
Notice that, for the Biot-Savart kernel on the whole space $\mathbb{R}^2$ $$\tilde{K}(x)=\frac{1}{2\pi}\frac{x^{\perp}}{|x|^2},$$  the matrix field $\tilde{V}$ such that $\tilde{K}=\nabla\cdot\tilde{V}$ can be chosen explicitly 
$$V(x)=\frac{1}{2\pi}\left(\begin{array}{cc}-\arctan\left(\frac{x_1}{x_2}\right)& 0\\ 0&\arctan\left(\frac{x_2}{x_1}\right)\end{array}\right).$$
\end{remark}

One could also consider collision-like interactions, as mentioned in \cite{JW18}. Let ${\phi\in L^1}$ be a function on the torus, $M$ be a smooth antisymmetric matrix field and consider the kernel ${K=\nabla\cdot(M\mathds{1}_{\phi(x)\leq0})}$. By construction, $K$ is the divergence of a $L^\infty$ matrix field, and since $M$ is antisymmetric $K$ is divergence free.

\begin{example}
Consider in dimension 2 the function $\phi:x\mapsto |x|^2-(2R)^2$  for a given radius $R>0$ and the matrix $$M=\left(\begin{array}{cc}0&-1\\1&0\end{array}\right),$$ which yield $$K(x)=2x^\perp\delta_{\phi(x)=0}.$$ This interaction kernel models particles, seen as balls of radius $R$, \modif{interacting via some form of collision}.
\end{example}

The well-posedness of the equations \eqref{o_Liou} and \eqref{part_liou} under Assumptions \ref{assumption_rho_0} and \ref{assumption_K} will be discussed respectively in Sections~\ref{sec:resultsPDE} and \ref{sec:nonsmoothrhon}. In particular we will see in Theorem~\ref{exist} that $\bar\rho_t$ is in $\mathcal C_\lambda^\infty(\mathbb{R}^+\times\mathds{T}^d)$.

The comparison between the law of the system of $N$ interacting particles and the law of $N$ independent particles satisfying the non-linear equation \eqref{o_Lang} is stated in terms of relative entropy. 


\begin{definition} 
Let $\mu$ and $\nu$ be two probability \modif{densities} on $\mathds{T}^{dN}$. We consider the rescaled relative entropy
\begin{equation}\label{rel_ent_gen}
\mathcal{H}_N(\nu,\mu)=\left\{
\begin{array}{ll}
\frac{1}{N}\mathbb{E}_\mu\left(\modif{\frac{\nu}{\mu}\log\frac{\nu}{\mu}}\right)\text{ if }\nu\ll\mu,\\
+\infty\text{ otherwise.}
\end{array}
\right.
\end{equation}
\end{definition}

\modif{For the sake of conciseness, for all $k\in\mathbb N$ and $t\geq 0$, we denote $\rho_N(t):\textbf{x}\in\mathds{T}^{dN}\mapsto\rho_N(t,\textbf{x})$ and $\bar{\rho}_N(t):\textbf{x}\in\mathds{T}^{dN}\mapsto\bar{\rho}_t^{\otimes N}(\textbf{x})$.} The main result is the following

\begin{theorem}\label{thm_prop}
Under Assumptions \ref{assumption_rho_0} and \ref{assumption_K}, there are constants $C_1$, $C_2$ and $C_3$ such that for all $N\in\mathbb N$ and all exchangeable density probability $\rho_N(0) \in \mathcal C_{>0}^\infty(\mathds{T}^{dN})$ there exists a weak solution $\rho_N$ of \eqref{part_liou} such that for all $t\geq0$
\begin{equation}
\mathcal{H}_N(\rho_N(t),\bar{\rho}_N(t))\leq C_1e^{-C_2t}\mathcal{H}_N(\rho_N(0),\bar{\rho}_N(0))+\frac{C_3}{N}
\end{equation}
\end{theorem}

In particular, if $\rho_N(0)=\bar\rho_N(0)$, the first term of the right-hand side vanishes, and this property has been called entropic propagation of chaos, see for example \cite{HM14}.

%
%

\subsection{Strong propagation of chaos}

We show that Theorem~\ref{thm_prop} yields strong propagation of chaos, uniform in time. For $\mu$ and $\nu$  two probability measures on $\mathds{T}^{dk}$, denote by $\Pi\left(\mu,\nu\right)$ the set of couplings of $\mu$ and $\nu$, i.e. the set of probability measures $\Gamma$ on $\mathds{T}^{dk}\times\mathds{T}^{dk}$ with $\Gamma(A\times \mathds{T}^{dk}) = \mu(A)$ and $\Gamma(\mathds{T}^{dk}\times A ) = \nu(A)$ for all Borel set $A$ of $\mathds{T}^{dk}$. Let us define the usual $L^2$-Wasserstein distance by 
\begin{align*}
\mathcal{W}_2\left(\mu,\nu\right)=\left(\inf_{\Gamma\in\Pi\left(\mu,\nu\right)}\int_{\mathds{T}^{dk}}d_{\mathds{T}^{dk}}(x,y)^2\Gamma\left(dxdy \right)\right)^{1/2},
\end{align*}
where $d_{\mathds{T}^{dk}}$ is the usual distance on the torus. For $\textbf{x}=(x_i)_{i\in\llbracket 1,N\rrbracket} \in \mathds{T}^{dN}$, we write $\pi(\textbf{x})=\frac1N\sum_{i=1}^N \delta_{x_i}$ the associated empirical measure.


\begin{corollary}\label{coro_wass}
Under assumptions \ref{assumption_rho_0} and \ref{assumption_K}, assuming moreover that $\rho_N(0)=\bar\rho_N(0)$, there is a constant $C$ such that for all $k\leq N \in \mathbb N$ and all $t\geq 0$, 
\begin{equation*}\|\rho_N^k(t) -\bar \rho_k(t)\|_{L^1} + 
\mathcal{W}_2\left(\rho_{N}^{k}(t),\bar{\rho}_k(t)\right)\leq C\left(\left\lfloor\frac{N}{k}\right\rfloor\right)^{-\frac{1}{2}}
\end{equation*}
and
\[\mathbb E_{\rho_N(t)} \left(\mathcal W_2(\pi(\textbf{X}),\modif{\bar\rho_t})\right) \leqslant C \alpha(N)\]
where $\alpha(N)=N^{-1/2}\ln(1+N)$ if $d=2$ and $\alpha(N)=N^{-1/d}$ if $d>2$.
\end{corollary}
As shown in \cite{BGV07}, the last result yields confidence interval in uniform norm when estimating  $\bar\rho_t$ with $\pi(\textbf{X}_t^N)$ convoluted to a smooth kernel.

We postpone the proof as it will rely on results shown later. It will however be a direct corollary of Theorem~\ref{thm_prop} and of the  logarithmic Sobolev inequality proven in Corollary~\ref{coro_logsob}, which is a crucial ingredient in the proof of  Theorem~\ref{thm_prop}.

%
%

\section{Preliminary work}\label{Preliminary_work}

%
%

\subsection{First results on the non-linear PDE}\label{sec:resultsPDE}

We have the following result concerning the solution of \eqref{o_Liou}.

\begin{theorem}\label{exist}
Under Assumption~\ref{assumption_K}, let $\mu_0\in\mathcal{C}_\lambda^\infty(\mathds{T}^d)$. Then the system
\begin{equation}\label{PDE_vortex}
\left\{
    \begin{array}{ll}
        \partial_t\bar{\rho}_t=-\nabla\cdot\left(\left(K\ast\bar{\rho}_t\right)\bar{\rho}_t\right)+\Delta\bar{\rho}_t, \text{ in } \mathbb{R}^+\times\mathds{T}^d \\
        \bar{\rho}_0=\mu_0,
    \end{array}
\right.
\end{equation}
has, in the class of bounded solutions, a unique solution $\bar{\rho}(t,x)\in\mathcal{C}_\lambda^\infty(\mathbb{R}^+\times\mathds{T}^d)$.
\end{theorem}

\begin{proof}
The first part of the theorem (existence, uniqueness and smoothness) can be proven by following closely the proof done by Ben-Artzi in \cite{Ben94}. For the sake of completeness, this is detailed in Appendix~\ref{Appendice_reg}. Note that a similar result has also been recently proven in \cite{Wyn23}, where the $\mathcal{C}^k$ regularity of $\bar{\rho}_t$ for any given $k$ and any given $t$ is shown. The proof relies heavily on the fact that the kernel $K$ is divergence free, that the convolution operation tends to keep the regularity of the most regular term, and that the Fokker-Planck equation has a smoothing effect.

Let us now prove the second part of the result, namely the time uniform bounds on $\bar{\rho}_t$. Assume that $\mu_0\in\mathcal{C}_\lambda^\infty(\mathds{T}^d)$, which by definition implies $\frac{1}{\lambda}\leq\mu_0\leq \lambda$,
and consider  $\bar{\rho}_t$ the unique solution of \eqref{PDE_vortex}.
We start by proving that $K\ast\bar{\rho}_t$ is in $\mathcal{C}^\infty$. By definition
\begin{align*}
K\ast\bar{\rho}_t(x)=\int_{\mathds{T}^d}K(x-y)\bar{\rho}_t(y)dy=-\int_{\mathds{T}^d}K(y)\bar{\rho}_t(x-y)dy.
\end{align*}
Then
\begin{align*}
K\ast\bar{\rho}_t(x)=-\int_{\mathds{T}^d}\nabla\cdot V(y)\bar{\rho}_t(x-y)dy=-\int_{\mathds{T}^d}V(y)\nabla_y\bar{\rho}_t(x-y)dy.
\end{align*}
Since   $V\in L^{\infty}(\mathds{T}^d)$ and  $\bar{\rho}\in\mathcal{C}^\infty(\mathbb{R}^+\times\mathds{T}^d)$, we   easily deduce that $K\ast \bar{\rho}$, as well as all its derivatives, are Lipschitz continuous on $[0,T]\times\mathds{T}^d$ for all $T>0$. Hence $K\ast \bar{\rho}$ is $\mathcal{C}^\infty$. Moreover, using that $\nabla\cdot K=0$ (in the sense of distribution), we immediately get that $\nabla \cdot (K \ast \bar \rho_t) = 0$ for all $t\geqslant 0$.

For $t\geqslant 0$ and $x\in\mathds{T}^d$, 
consider $Z_s$ the  strong solution of the following stochastic differential equation for $s\in[0,t]$
\begin{equation*}
 dZ_s=\sqrt{2}dB_s-K\ast\bar{\rho}_{t-s}(Z_s)ds,\ Z_0=x
\end{equation*}
which exists, is unique  and non-explosive since $K\ast \bar\rho_{t-s}$ is smooth and bounded. Then 
\[\bar{\rho}(t,x) = \mathbb E_x\left( \bar\rho_0(Z_t)\right).\]
The bounds on $\bar\rho_t$ follow.
\end{proof}

%
%

\subsection{Higher order estimates}

We have already established that $\bar{\rho}_t$ is bounded uniformly in time. In this section, we extend this result to all its derivatives.

\begin{lemma}\label{Bornes_derivees}
For all $n\geqslant 1$ and $\ \alpha_1,...,\alpha_n\in\llbracket1,d\rrbracket$, there exist $ C_n^u, C_n^\infty>0$ such that for all $t\geqslant 0$,
\begin{equation*}
\|\partial_{\alpha_1,...,\alpha_n}\bar{\rho}_t\|_{L^\infty}\leq C_n^u \quad\text{ and }\quad \int_0^t\|\partial_{\alpha_1,...,\alpha_n}\bar{\rho}_s\|^2_{L^\infty} ds\leq C_n^\infty
\end{equation*}
\end{lemma}

\begin{proof}Thanks to Morrey's inequality and  Sobolev embeddings, it is sufficient to prove such bounds in the Sobolev space $H^m$ for all $m$, in other words it is sufficient to prove similar bounds for  $\|\partial_{\alpha_1,...,\alpha_n}\bar{\rho}_s\|^2_{L^2}$ for all multi-indexes $\alpha$. The proof is by induction on the order of the derivatives, we only detail the first iterations. 
We write $f=\nabla\cdot\left(\left(K\ast \bar{\rho}_t\right)\bar{\rho}_t\right)=(K\ast\bar\rho_t)\cdot\nabla \bar\rho_t$.
\begin{description}

\item[Integrated bound for $\|\nabla \bar{\rho}_t\|_{L^2}^2$.] We have
\begin{align*}
\frac{1}{2}\frac{d}{dt}\int_{\mathds{T}^d}|\bar{\rho}_t|^2=\int_{\mathds{T}^d}\bar{\rho}_t\partial_t\bar{\rho}_t=\int_{\mathds{T}^d}\bar{\rho}_t\Delta\bar{\rho}_t-\int_{\mathds{T}^d}\bar{\rho}_t f.
\end{align*}
On the one hand,
\begin{align*}
\int_{\mathds{T}^d}\bar{\rho}_t\Delta\bar{\rho}_t=-\int_{\mathds{T}^d}|\nabla \bar{\rho}_t|^2.
\end{align*}
On the other hand,
\begin{align*}
\int_{\mathds{T}^d}\bar{\rho}_t f=\int_{\mathds{T}^d}\bar{\rho}_t \nabla\cdot\left(\left(K\ast \bar{\rho}_t\right)\bar{\rho}_t\right)=-\int_{\mathds{T}^d}\nabla\bar{\rho}_t\cdot\left(K\ast \bar{\rho}_t\right)\bar{\rho}_t=-\int_{\mathds{T}^d}\bar{\rho}_t f=0.
\end{align*}
Hence,
\begin{align*}
\frac{1}{2}\frac{d}{dt}\|\bar{\rho}_t\|^2_{L^2}+\|\nabla \bar{\rho}_t\|^2_{L^2}=0.
\end{align*}
By integrating the equality above, we get $\int_0^t\|\nabla \bar{\rho}_t\|^2_{L^2}=\frac{\|\bar{\rho}_0\|^2_{L^2}-\|\bar{\rho}_t\|^2_{L^2}}{2}\leq \frac{\lambda^2}{2}=C^{\infty}_1$.

\item[Integrated bound for $\|\partial_{\alpha_1,\alpha_2} \bar{\rho}_t\|_{L^2}^2$ and uniform  bound for $\|\nabla \bar{\rho}_t\|_{L^2}^2$.] Similarly, we calculate
\begin{align*}
\frac{1}{2}\frac{d}{dt}\int_{\mathds{T}^d}|\partial_{\alpha_1} \bar{\rho}_t|^2=\int_{\mathds{T}^d}\partial_{\alpha_1} \bar{\rho}_t\partial_{\alpha_1} (\partial_t\bar{\rho}_t)&=\int_{\mathds{T}^d}\partial_{\alpha_1} \bar{\rho}_t\partial_{\alpha_1} \left(\Delta\bar{\rho}_t-f\right)\\
&=-\sum_{\alpha_2}\int_{\mathds{T}^d}|\partial_{\alpha_1,\alpha_2} \bar{\rho}_t|^2+\int_{\mathds{T}^d}\partial_{\alpha_1,\alpha_1}\bar{\rho}_t f .
\end{align*}
Bounding
\begin{align*}
\int_{\mathds{T}^d}\partial_{\alpha_1,\alpha_1}\bar{\rho}_t f\leq \|\partial_{\alpha_1,\alpha_1}\bar{\rho}_t \|_{L^2} \|f \|_{L^2}\leq \frac{1}{2}\sum_{\alpha_2}\|\partial_{\alpha_1,\alpha_2}\bar{\rho}_t \|^2_{L^2}+\frac{1}{2}\|f \|^2_{L^2},
\end{align*}
and 
\begin{align*}
\|f \|^2_{L^2}=\int_{\mathds{T}^d}\left|\sum_{\gamma=1}^d\left(K_\gamma\ast \bar{\rho}_t\right)\partial_{\gamma}\bar{\rho}_t\right|^2\leq \|K\ast \bar{\rho}_t\|_{L^\infty}^2\|\nabla\bar{\rho}_t\|^2_{L^2}\leq \|K\|_{L^1}^2\|\bar{\rho}_t\|_{L^\infty}^2\|\nabla\bar{\rho}_t\|^2_{L^2},
\end{align*}
where we used Young's convolution inequality, we get
\begin{align*}
\frac{1}{2}\frac{d}{dt}\|\partial_{\alpha_1} \bar{\rho}_t\|^2_{L^2}+\frac{1}{2}\sum_{\alpha_2}\|\partial_{\alpha_1,\alpha_2}\bar{\rho}_t\|^2_{L^2}\leq\frac{1}{2} \|K\|_{L^1}^2\|\bar{\rho}_t\|_{L^\infty}^2\|\nabla\bar{\rho}_t\|^2_{L^2}.
\end{align*}
By integrating the equality above and using Theorem~\ref{exist}, we get
\begin{align*}
\frac{\|\partial_{\alpha_1} \bar{\rho}_t\|^2_{L^2}-\|\partial_{\alpha_1} \bar{\rho}_0\|^2_{L^2}}{2}+\frac{1}{2}\int_{0}^t\sum_{\alpha_2}\|\partial_{\alpha_1,\alpha_2}\bar{\rho}_s\|^2_{L^2}ds\leq&\frac{1}{2} \|K\|_{L^1}^2\lambda^2\int_0^t\|\nabla\bar{\rho}_s\|^2_{L^2}ds\\
\leq&\frac{1}{2} \|K\|_{L^1}^2\lambda^2 C^{\infty}_1.
\end{align*}
This provides  both the existence of $C^\infty_{2}$ such that for all $t\geq0$, $\int_{0}^t\|\partial_{\alpha_1,\alpha_2}\bar{\rho}_s\|^2_{L^2}ds\leq C^\infty_{2}$, and the existence of $C^u_{1}$ such that for all $t\geq0$, $\|\partial_{\alpha_1} \bar{\rho}_t\|^2_{L^2}\leq C^u_{1}$.

\item[Integrated bound bound for $\|\partial_{\alpha_1,\alpha_2,\alpha_3} \bar{\rho}_t\|_{L^2}^2$ and uniform  bound for $\|\partial_{\alpha_1,\alpha_2} \bar{\rho}_t\|_{L^2}^2$.] We have
\begin{align*}
\partial_{\alpha} f=\sum_{\gamma}(\partial_\alpha K_\gamma\ast \bar{\rho}_t)\partial_\gamma\bar{\rho}_t+\sum_{\gamma}( K_\gamma\ast \bar{\rho}_t)\partial_{\alpha,\gamma}\bar{\rho}_t,
\end{align*}
and
\begin{multline*}
\partial_\alpha K_\gamma\ast \bar{\rho}_t=\int_{\mathds{T}^d}\partial_\alpha K_\gamma(x-y)\bar{\rho}_t(y)dy=-\int_{\mathds{T}^d}\partial_\alpha K_\gamma(y)\bar{\rho}_t(x-y)dy=-\int_{\mathds{T}^d}K_\gamma(y)\partial_\alpha\bar{\rho}_t(x-y)dy\\
=-\sum_{\beta}\int_{\mathds{T}^d}V_{\gamma,\beta}(y)\partial_{\alpha,\beta}\bar{\rho}_t(x-y)dy=\sum_{\beta}V_{\gamma,\beta}\ast \partial_{\alpha,\beta}\bar{\rho}_t.
\end{multline*}
Hence
\begin{align*}
\sum_\gamma\left(\partial_\alpha K_\gamma\ast \bar{\rho}_t\right)\partial_\gamma\bar{\rho}_t=&\sum_\gamma\left(\sum_{\beta}V_{\gamma,\beta}\ast \partial_{\alpha,\beta}\bar{\rho}_t\right)\partial_\gamma\bar{\rho}_t\\
=&\left(V\ast\partial_\alpha\nabla\bar{\rho}_t\right)\nabla\bar{\rho}_t,
\end{align*}
and thus
\begin{align*}
\|\sum_\gamma\left(\partial_\alpha K_\gamma\ast \bar{\rho}_t\right)\partial_\gamma\bar{\rho}_t\|_{L^2}\leq\|V\ast\partial_\alpha\nabla\bar{\rho}_t\|_{L^\infty}\|\nabla\bar{\rho}_t\|_{L^2}\leq \|V\|_{L^\infty}\|\partial_\alpha\nabla\bar{\rho}_t\|_{L^1}\|\nabla\bar{\rho}_t\|_{L^2}.
\end{align*}
Therefore
\begin{align*}
\|\partial_{\alpha} f\|_{L^2}^2\leq 2\|V\|_{L^\infty}^2\|\partial_\alpha\nabla\bar{\rho}_t\|_{L^1}^2\|\nabla\bar{\rho}_t\|^2_{L^2}+2\|K\|_{L^1}^2\|\bar{\rho}_t\|_{L^\infty}^2\|\partial_{\alpha}\nabla \bar{\rho}_t\|^2_{L^2}.
\end{align*}
Similarly to previous computations,
\begin{align*}
\frac{1}{2}\frac{d}{dt}\int_{\mathds{T}^d}|\partial_{\alpha_1,\alpha_2}\bar{\rho}_t|^2=&\int_{\mathds{T}^d}\partial_{\alpha_1,\alpha_2}\bar{\rho}_t\partial_{\alpha_1,\alpha_2}\left(\Delta\bar{\rho}_t-f\right)\\
=&-\sum_{\alpha_3}\int_{\mathds{T}^d}|\partial_{\alpha_1,\alpha_2,\alpha_3}\bar{\rho}_t|^2+\int_{\mathds{T}^d}\partial_{\alpha_1,\alpha_2,\alpha_2}\bar{\rho}_t\partial_{\alpha_1}f\\
\leq&-\sum_{\alpha_3}\|\partial_{\alpha_1,\alpha_2,\alpha_3}\bar{\rho}_t\|^2_{L^2}+\|\partial_{\alpha_1,\alpha_2,\alpha_2}\bar{\rho}_t\|_{L^2}\|\partial_{\alpha_1}f\|_{L^2}\\
\leq&-\sum_{\alpha_3}\|\partial_{\alpha_1,\alpha_2,\alpha_3}\bar{\rho}_t\|^2_{L^2}+\frac{1}{2}\sum_{\alpha_3}\|\partial_{\alpha_1,\alpha_2,\alpha_3}\bar{\rho}_t\|_{L^2}^2\\
&+\|V\|_{L^\infty}^2\|\partial_{\alpha_1}\nabla\bar{\rho}_t\|_{L^1}^2\|\nabla\bar{\rho}_t\|^2_{L^2}+\|K\|_{L^1}^2\|\bar{\rho}_t\|_{L^\infty}^2\|\partial_{\alpha_1}\nabla \bar{\rho}_t\|^2_{L^2}\\
\leq&-\frac{1}{2}\sum_{\alpha_3}\|\partial_{\alpha_1,\alpha_2,\alpha_3}\bar{\rho}_t\|_{L^2}^2+\|V\|_{L^\infty}^2\|\nabla\bar{\rho}_t\|^2_{L^2}\|\partial_{\alpha_1}\nabla\bar{\rho}_t\|_{L^2}^2\\
&\qquad+\|K\|_{L^1}^2\|\bar{\rho}_t\|_{L^\infty}^2\|\partial_{\alpha_1}\nabla \bar{\rho}_t\|^2_{L^2},
\end{align*}
and thus
\begin{align*}
\frac{1}{2}\frac{d}{dt}\|\partial_{\alpha_1,\alpha_2}\bar{\rho}_t\|^2_{L^2}+\frac{1}{2}\sum_{\alpha_3}\|\partial_{\alpha_1,\alpha_2,\alpha_3}\bar{\rho}_t\|_{L^2}^2\leq&\|V\|_{L^\infty}^2\|\partial_{\alpha_1}\nabla\bar{\rho}_t\|_{L^2}^2\|\nabla\bar{\rho}_t\|^2_{L^2}\\
&\qquad+\|K\|_{L^1}^2\|\bar{\rho}_t\|_{L^\infty}^2\|\partial_{\alpha_1}\nabla \bar{\rho}_t\|^2_{L^2}.
\end{align*}
Integrating over time, and using Theorem~\ref{exist}
\begin{align*}
\frac{\|\partial_{\alpha_1,\alpha_2}\bar{\rho}_t\|^2_{L^2}-\|\partial_{\alpha_1,\alpha_2}\bar{\rho}_0\|^2_{L^2}}{2}+&\frac{1}{2}\sum_{\alpha_3}\int_{0}^t\|\partial_{\alpha_1,\alpha_2,\alpha_3}\bar{\rho}_s\|_{L^2}^2ds\\
\leq&\|V\|_{L^\infty}^2dC^u_1\int_{0}^t\|\partial_{\alpha_1}\nabla\bar{\rho}_s\|_{L^2}^2ds+\|K\|_{L^1}^2\lambda^2\int_{0}^t\|\partial_{\alpha_1}\nabla \bar{\rho}_s\|^2_{L^2}ds\\
\leq&d\left(d\|V\|_{L^\infty}^2C^u_1+\|K\|_{L^1}^2\lambda^2\right)C^\infty_2
\end{align*}
This provides  both the existence of $C^\infty_{3}$ such that for all $t\geq0$, $\int_{0}^t\|\partial_{\alpha_1,\alpha_2,\alpha_3}\bar{\rho}_s\|^2_{L^2}ds\leq C^\infty_{3}$, and the existence of $C^u_{2}$ such that for all $t\geq0$, $\|\partial_{\alpha_1,\alpha_2} \bar{\rho}_t\|^2_{L^2}\leq C^u_{2}$.
\end{description}
The proof is then by induction on the order of derivative, iterating the same method.
\end{proof}

%
%

\subsection{Logarithmic Sobolev inequality}

We now establish a logarithmic Sobolev inequality (LSI) for $\bar{\rho}_t$ solution of \eqref{o_Liou}. To this end, we use the fact that the uniform distribution $u$ on $\mathds{T}^d$ satisfies a LSI and that $\bar{\rho}_t$ is bounded (below and above) uniformly in time. Recall the following  Holley-Stroock perturbation lemma, from \cite[Prop. 5.1.6]{BGL14}.


\begin{lemma}\label{BGL_6}
Assume that $\nu$ is a probability measure on $\mathds{T}^{d}$ satisfying a logarithmic Sobolev inequality with constant $C^{LS}_\nu$, i.e for all  $f\in\mathcal{C}^\infty_{>0}(\mathds{T}^{d})$,
\begin{equation*}
\mathrm{Ent}_\nu(f):=\int_{\mathds{T}^{d}}f\log f d\nu-\int_{\mathds{T}^{d}}fd\nu\log \left(\int_{\mathds{T}^{d}}f d\nu\right)\leq C^{LS}_\nu\int_{\mathds{T}^{d}}\frac{|\nabla f|^2}{f}d\nu.
\end{equation*}
Let $\mu$ be a probability measure with density $h$ with respect to $\nu$ such that, for some constant $\lambda>0$, $\frac{1}{\lambda}\leq h\leq \lambda$. Then $\mu$ satisfies a logarithmic Sobolev inequality with constant $C^{LS}_\mu=\lambda^2C^{LS}_\nu$, i.e for all    $f\in\mathcal{C}^\infty_{>0}(\mathds{T}^{d})$
\begin{equation*}
\mathrm{Ent}_\mu(f)\leq \lambda^2C^{LS}_\nu\int_{\mathds{T}^{d}}\frac{|\nabla f|^2}{f}d\mu.
\end{equation*}
\end{lemma}

We also know that the uniform distribution $u$ \modif{(i.e the Lebesgue measure)} on $\mathds{T}^d$ satisfies a LSI. See for instance Proposition 5.7.5 of \cite{BGL14}, or \cite{IG_notes} for a proof in dimension 1, the results in higher dimension being a consequence of tensorization properties.


\begin{lemma}\label{log_sob_unif}
Let $u$ be the uniform distribution on $\mathds{T}^{d}$ . Then $u$ satisfies a logarithmic Sobolev inequality : for all  $f\in\mathcal{C}^\infty_{>0}(\mathds{T}^{d})$
\begin{equation} 
\mathrm{Ent}_u(f) \leq\frac{1}{8\pi^2}\int_{\mathds{T}^d}\frac{|\nabla f|^2}{f}du
\end{equation} 
\end{lemma}

A direct consequence of Lemma~\ref{BGL_6}, Lemma~\ref{log_sob_unif} and the bounds on $\bar{\rho}_t$ given in Theorem~\ref{exist} is the following theorem, as well as its corollary. It establishes a uniform in time logarithmic Sobolev inequality for $\bar{\rho_t}$, crucial for the uniform control of the Fisher information appearing in the study of the dissipation of the entropy between the law of the particles system and the nonlinear ones.


\begin{theorem}\label{ISL_pour_rho}
Under Assumptions~\ref{assumption_rho_0} and \ref{assumption_K}, for all $t\geq0$ and all function $f\in\mathcal{C}^\infty_{>0}(\mathds{T}^{d})$,
\begin{equation*}
\mathrm{Ent}_{\bar{\rho}_t}(f)\leq \frac{\lambda^2}{8\pi^2}\int_{\mathds{T}^d}\frac{|\nabla f|^2}{f}d\bar{\rho}_t
\end{equation*}
\end{theorem}


\begin{corollary}\label{coro_logsob}
Under Assumptions~\ref{assumption_rho_0} and \ref{assumption_K}, for all $N\in\mathbb N$, $t\geq0$ and all probability density $\mu_N\in\mathcal C^{\infty}_{>0}(\mathds{T}^{dN})$,
\begin{equation*}
\mathcal{H}_N\left(\mu_N,\bar{\rho}_N(t)\right)\leq\frac{\lambda^2}{8\pi^2}\frac{1}{N}\sum_{i=1}^N\int_{\mathds{T}^{d}}\mu_N\left|\nabla_{x_i}\log\frac{\mu_N}{\bar{\rho}_N(t)}\right|^2
\end{equation*}
\end{corollary}


\begin{proof}
By tensorization of the logarithmic Sobolev inequality (see for instance Proposition 5.2.7 of \cite{BGL14}), since $\bar{\rho}$ satisfies a LSI with constant $\frac{\lambda^2}{8\pi^2}$, so does $\bar{\rho}_N$. Using Theorem~\ref{ISL_pour_rho} for $f=\frac{\mu_N}{\bar{\rho}_N}$ we thus get
\begin{align*}
\mathcal{H}_N(\mu_N,\bar{\rho}_N\modif{(t)})=\frac{1}{N}\text{Ent}_{\bar{\rho}_N\modif{(t)}}\left(\frac{\mu_N}{\bar{\rho}_N\modif{(t)}}\right)\leq \frac{\lambda^2}{8\pi^2}\frac{1}{N}\mathbb{E}_{\bar{\rho}_N\modif{(t)}}\left(\left|\nabla_x\frac{\mu_N}{\bar{\rho}_N\modif{(t)}}\right|^2\frac{\bar{\rho}_N\modif{(t)}}{\mu_N}\right).
\end{align*}
Hence the result.
\end{proof}

%
%

\section{Proofs of the main results}\label{Section_Proof}

From now on and up to Section~\ref{sec:nonsmoothrhon} (excluded), in addition to Assumptions~\ref{assumption_rho_0} and \ref{assumption_K}, we suppose that there  exists a solution   ${\rho_N \in \mathcal{C}^{\infty}_{>0}(\mathbb{R}^+\times\mathds{T}^{dN})}$ of \eqref{part_liou}. This justifies the validity of the various calculations conducted in this part of the proof.
The question to lift this assumption (by taking a limit in a regularized problem) is addressed in Section~\ref{sec:nonsmoothrhon}

%
%

\subsection{Time evolution of the relative entropy}

 We write 
\[\mathcal H_N(t) = \mathcal{H}_N(\rho_N(t),\bar{\rho}_N(t))\,,\qquad \mathcal I_N(t) = \frac{1}{N}\sum_i\int_{\mathds{T}^{dN}}\rho_N(t)\left|\nabla_{x_i}\log\frac{\rho_N(t)}{\bar{\rho}_N(t)}\right|^2\modif{d\textbf{x}}.\]
as short hands for the relative entropy and relative Fisher information.
We start by calculating the time evolution of the relative entropy.


\begin{lemma}\label{lem:dissipation}
For all $t\geqslant 0$, 
\begin{equation}\label{maj_H}
\frac{d}{dt}\mathcal{H}_N(t)\leq A_N(t)+\frac{1}{2}B_N(t)-\frac{1}{2} \mathcal I_N(t),
\end{equation}
with
\begin{align*}
A_N(t):=&\frac{1}{N^2}\sum_{i,j}\int_{\mathds{T}^{dN}}\rho_N\left(V(x_i-x_j)-V\ast \modif{\bar{\rho}_t}(x_i)\right):\frac{\nabla_{x_i}^2\bar{\rho}_N}{\bar{\rho}_N}\modif{d\textbf{x}}\\
B_N(t):=&\frac{1}{N}\sum_{i}\int_{\mathds{T}^{dN}}\rho_N\frac{\left|\nabla_{x_i}\bar{\rho}_N\right|^2}{\bar{\rho}_N^2}\left|\frac{1}{N}\sum_jV(x_i-x_j)-V\ast \modif{\bar{\rho}_t}(x_i)\right|^2_\modif{f}\modif{d\textbf{x}}.
\end{align*}
\modif{Here, $|\cdot|_f^2$ denotes the sum of the square of the coefficients of the matrix.}
\end{lemma}


\begin{proof}
It has been shown, in \cite{JW18}, that
\begin{align*}
\frac{d}{dt}\mathcal{H}_N(t)\leq- \mathcal I_N(t)-\frac{1}{N^2}\sum_{i,j}\int_{\mathds{T}^{dN}}\rho_N\left(K(x_i-x_j)-K\ast \modif{\bar{\rho}_t}(x_i)\right)\cdot\nabla_{x_i}\log\bar{\rho}_N \modif{d\textbf{x}},
\end{align*}
with
\begin{align*}
-\frac{1}{N^2}\sum_{i,j}\int_{\mathds{T}^{dN}}\rho_N&\left(K(x_i-x_j)-K\ast \modif{\bar{\rho}_t}(x_i)\right)\cdot\nabla_{x_i}\log\bar{\rho}_N \modif{d\textbf{x}}\\
=&\frac{1}{N^2}\sum_{i,j}\int_{\mathds{T}^{dN}}\rho_N\left(V(x_i-x_j)-V\ast \modif{\bar{\rho}_t}(x_i)\right):\frac{\nabla_{x_i}^2\bar{\rho}_N}{\bar{\rho}_N}\modif{d\textbf{x}}\\
&+\frac{1}{N^2}\sum_{i,j}\int_{\mathds{T}^{dN}}\left(V(x_i-x_j)-V\ast \modif{\bar{\rho}_t}(x_i)\right):\nabla_{x_i}\bar{\rho}_N\otimes\nabla_{x_i}\frac{\rho_N}{\bar{\rho_N}}\modif{d\textbf{x}}.
\end{align*}
Let us consider the latter term
\begin{align*}
\frac{1}{N^2}\sum_{i,j}\int_{\mathds{T}^{dN}}&\left(V(x_i-x_j)-V\ast \modif{\bar{\rho}_t}(x_i)\right):\nabla_{x_i}\bar{\rho}_N\otimes\nabla_{x_i}\frac{\rho_N}{\bar{\rho_N}}\modif{d\textbf{x}}\\
&=\frac{1}{N}\sum_{i}\sum_{\alpha,\beta}\int_{\mathds{T}^{dN}}\left(\frac{1}{N}\sum_jV(x_i-x_j)-V\ast \modif{\bar{\rho}_t}(x_i)\right)_{\alpha,\beta}\left(\nabla_{x_i}\bar{\rho}_N\right)_\alpha\left(\nabla_{x_i}\frac{\rho_N}{\bar{\rho_N}}\right)_\beta \modif{d\textbf{x}}.
\end{align*}
Let
\begin{align*}
y^i_{\beta}:=\left(\nabla_{x_i}\frac{\rho_N}{\bar{\rho_N}}\right)_\beta\frac{\bar{\rho}_N}{\sqrt{\rho_N}},\ \ \ z^i_\alpha:=\left(\nabla_{x_i}\bar{\rho}_N\right)_\alpha\frac{\sqrt{\rho_N}}{\bar{\rho}_N},\ \text{ and }\ x^i_{\alpha,\beta}:=\left(\frac{1}{N}\sum_jV(x_i-x_j)-V\ast \bar{\rho}(x_i)\right)_{\alpha,\beta},
\end{align*}
then, using $xy\leq\frac{x^2}{2}+\frac{y^2}{2}$ for all $x,y\in\mathbb{R}$, 
\begin{align*}
\sum_{\alpha,\beta}x^i_{\alpha,\beta}z^i_\alpha y^i_{\beta}=\sum_\beta y^i_{\beta}\left(\sum_\alpha x^i_{\alpha,\beta}z^i_\alpha\right)\leq\frac{1}{2}\sum_\beta (y^i_{\beta})^2+\frac{1}{2}\sum_\beta\left(\sum_\alpha x^i_{\alpha,\beta}z^i_\alpha\right)^2,
\end{align*}
and thus, using the Cauchy-Schwarz inequality,
\begin{align*}
\sum_{\alpha,\beta}x^i_{\alpha,\beta}z^i_\alpha y^i_{\beta}\leq&\frac{1}{2}\sum_\beta (y^i_{\beta})^2+\frac{1}{2}\sum_\beta \left( \sum_\alpha (x^i_{\alpha,\beta})^2\right) \left(\sum_\alpha (z^i_\alpha)^2\right)\\
=&\frac{1}{2}\sum_\beta (y^i_{\beta})^2+\frac{1}{2}\left(\sum_\alpha (z^i_\alpha)^2\right) \left( \sum_{\alpha,\beta} (x^i_{\alpha,\beta})^2\right).
\end{align*}
Hence
\begin{align*}
\frac{1}{N^2}&\sum_{i,j}\int_{\mathds{T}^{dN}}\left(V(x_i-x_j)-V\ast \modif{\bar{\rho}_t}(x_i)\right):\nabla_{x_i}\bar{\rho}_N\otimes\nabla_{x_i}\frac{\rho_N}{\bar{\rho_N}}\modif{d\textbf{x}}\\
\leq&\frac{1}{2N}\sum_{i}\int_{\mathds{T}^{dN}}\frac{\bar{\rho}_N^2}{\rho_N}\left|\nabla_{x_i}\frac{\rho_N}{\bar{\rho_N}}\right|^2\modif{d\textbf{x}}+\frac{1}{2N}\sum_{i}\int_{\mathds{T}^{dN}}\rho_N\frac{\left|\nabla_{x_i}\bar{\rho}_N\right|^2}{\bar{\rho}_N^2}\left|\frac{1}{N}\sum_jV(x_i-x_j)-V\ast \modif{\bar{\rho}_t}(x_i)\right|^2_\modif{f}\modif{d\textbf{x}}\\
=&\frac{1}{2}\mathcal I_N(t) +\frac{1}{2N}\sum_{i}\int_{\mathds{T}^{dN}}\rho_N\frac{\left|\nabla_{x_i}\bar{\rho}_N\right|^2}{\bar{\rho}_N^2}\left|\frac{1}{N}\sum_jV(x_i-x_j)-V\ast \modif{\bar{\rho}_t}(x_i)\right|^2_\modif{f}\modif{d\textbf{x}}.
\end{align*}
This yields the desired result.
\end{proof}

%
%

\subsection{Change of reference measure and Law of Large Number}

We now state three general results which will be useful in order to control the error terms $A_N$ and $B_N$ defined in Lemma~\ref{lem:dissipation}. The first one will be used to perform a change of measure from $\rho_N$ to $\bar \rho_N$.


\begin{lemma}\label{changement_de_loi}
\modif{Let $N\in\mathbb N$}. For two probability densities $\mu$ and $\nu$ on \modif{$\mathds{T}^{dN}$,} 
and any $\Phi\in L^{\infty}(\modif{\mathds{T}^{dN}})$  and  $\eta>0$,
\begin{equation*}
\mathbb{E}^\mu\Phi\leq\eta\mathcal{H}_N(\mu,\nu)+\frac{\eta}{N}\log\mathbb{E}^\nu e^{N\Phi/\eta}.
\end{equation*}
\end{lemma}


\begin{proof}
Define
\begin{equation*}
f=\frac{1}{\theta}e^{N\Phi/\eta}\nu,\hspace{1cm}\theta=\int_{\mathds{T}^{dN}}e^{ N\Phi/\eta}\nu \modif{d\textbf{x}}.
\end{equation*}
Notice $f$ is a probability density. By convexity of the entropy
\begin{equation*}
\frac{1}{N}\int_{\mathds{T}^{dN}}\mu\log f\modif{d\textbf{x}}\leq\frac{1}{N}\int_{\mathds{T}^{dN}}\mu\log \mu \modif{d\textbf{x}}.
\end{equation*}
On the other hand
\begin{equation*}
\frac{1}{N}\int_{\mathds{T}^{dN}}\mu\log f\modif{d\textbf{x}}=\frac{1}{\eta }\int_{\mathds{T}^{dN}}\mu\Phi \modif{d\textbf{x}}+\frac{1}{N}\int_{\mathds{T}^{dN}}\mu\log \nu \modif{d\textbf{x}}-\frac{\log \theta}{N}.
\end{equation*}
\end{proof}

The next two statements are crucial theorems of \cite{JW18}.

\begin{theorem}\label{borner_A}[\modif{Theorem~3 of \cite{JW18}}]
Consider any probability measure $\mu$ on $\mathds{T}^{d}$ and a scalar function $\psi\in L^{\infty}(\mathds{T}^d\times\mathds{T}^d)$ with $\|\psi\|_{L^\infty}<\frac{1}{2\modif{e}}$ and such that for all  $z\in\mathds{T}^d$, $\int_{\mathds{T}^{d}}\psi(z,x)\mu(dx)=0$. Then
\begin{equation}\label{borne_A}
\int_{\mathds{T}^{dN}}\exp\Big(\frac{1}{N}\sum_{j_1,j_2=1}^N\psi(x_1,x_{j_1})\psi(x_1,x_{j_2})\Big)\mu^{\otimes N}\modif{d\textbf{x}}\leq C=2\left(1+\frac{10\alpha}{(1-\alpha)^3}+\frac{\beta}{1-\beta}\right),
\end{equation}
where 
\begin{align*}
\alpha=\left(\modif{e}\|\psi\|_{L^\infty}\right)^4<1\text{ , }\ \beta=\left(\sqrt{2\modif{e}}\|\psi\|_{L^\infty}\right)^4<1.
\end{align*}
\end{theorem}

The second one is a nice improvement of the usual level two large deviations bound for i.i.d. random variables.

\begin{theorem}\label{borner_B}[\modif{Theorem~4 of \cite{JW18}}]
Consider any probability measure $\mu$ on $\mathds{T}^{d}$ and  $\phi\in L^{\infty}(\mathds{T}^d\times\mathds{T}^d)$ with
\begin{equation}\label{gamma}
\gamma:=\left(1600^2+36e^4\right)\Big(\sup_{p\geq1}\frac{\|\sup_{z}|\phi(\cdot,z)|\|_{L^p(\mu))}}{p}\Big)^2<1.
\end{equation}
Assume that $\phi$ satisfies the following cancellations
\begin{equation*}
\forall z\in\mathds{T}^d,\quad \int_{\mathds{T}^{d}}\phi(x,z)\mu(dx)=0 = \int_{\mathds{T}^{d}}\phi(z,x)\mu(dx)\,.
\end{equation*}
Then, for all $N\in\mathbb N$,
\begin{equation}\label{borne_B}
\int_{\mathds{T}^{dN}}\exp\Big(\frac{1}{N}\sum_{i,j=1}^N\phi(x_i,x_{j})\Big)\mu^{\otimes N}\modif{d\textbf{x}}\leq \frac{2}{1-\gamma}<\infty.
\end{equation}
\end{theorem}

%
%

\subsection{Bounding the error terms}

\begin{lemma}\label{lem:AB}
The terms $A_N$ and $B_N$ introduced in Lemma~\ref{lem:dissipation} are such that

\[A_N(t) +\frac12 B_N(t) \leqslant C\left(\mathcal H_N(t) + \frac1N\right)\]
with 
\[C = \hat{C}_1\modif{\lambda}\modif{d}\|\nabla^2\bar{\rho}_t\|_{L^\infty}\|V\|_{L^\infty}+\hat{C}_2\modif{\lambda^2d^2} \|V\|_{L^\infty}^2\|\nabla\bar{\rho}_t\|_{L^\infty}^2\]
where $\hat{C}_1,\hat C_2$ are universal constants.
\end{lemma}

\begin{proof}
Recall from Theorem~\ref{exist} that $\modif{\bar{\rho}_t}\in\mathcal C^\infty_\lambda(\mathbb T^d)$ for all $t\geq0$. We first bound $B_N$. \modif{For $(X^i_t)_i$ given in \eqref{o_Lang_part}}, we have
\begin{align*}
B_N=&\frac{1}{N}\sum_{i}\int_{\mathds{T}^{dN}}\rho_N\frac{\left|\nabla\modif{\bar{\rho}_t}\right|^2}{\modif{\bar{\rho}_t}^2}(x_i)\left|\frac{1}{N}\sum_jV(x_i-x_j)-V\ast \modif{\bar{\rho}_t}(x_i)\right|^2_\modif{f}\modif{d\textbf{x}}\\
=&\frac{1}{N}\sum_{i}\mathbb{E}\left(\left|\frac{\nabla\modif{\bar{\rho}_t}}{\modif{\bar{\rho}_t}}(X^i_t)\right|^2\left|\frac{1}{N}\sum_jV(X^i_t-X^j_t)-V\ast \modif{\bar{\rho}_t}(X^i_t)\right|^2_\modif{f}\right)\\
=&\frac{1}{N}\sum_{i}\sum_{\alpha,\beta=1}^d\mathbb{E}\left(\left|\frac{\nabla\modif{\bar{\rho}_t}}{\modif{\bar{\rho}_t}}(X^i_t)\right|^2\left(\frac{1}{N}\sum_jV_{\alpha,\beta}(X^i_t-X^j_t)-V_{\alpha,\beta}\ast \modif{\bar{\rho}_t}(X^i_t)\right)^2\right)\\
\leq&\frac{\lambda^2 \|\nabla\bar{\rho}_t\|^2_{L^\infty}}{N}\sum_{i}\sum_{\alpha,\beta=1}^d\mathbb{E}\left(\left(\frac{1}{N}\sum_jV_{\alpha,\beta}(X^i_t-X^j_t)-V_{\alpha,\beta}\ast \modif{\bar{\rho}_t}(X^i_t)\right)^2\right).
\end{align*}
We apply Lemma~\ref{changement_de_loi} to each
\begin{align*}
\Phi_{\alpha,\beta}=\left(\frac{1}{N}\sum_jV_{\alpha,\beta}(x_i-x_j)-V_{\alpha,\beta}\ast \modif{\bar{\rho}_t}(x_i)\right)^2,
\end{align*}
to get, for all $C_B>0$,
\begin{align*}
\mathbb{E}&\left(\left(\frac{1}{N}\sum_jV_{\alpha,\beta}(X^i_t-X^j_t)-V_{\alpha,\beta}\ast \modif{\bar{\rho}_t}(X^i_t)\right)^2\right)\\
&\leq C_B\mathcal{H}_N(t)+\frac{C_B}{N}\log\mathbb{E}\left(\exp\left(\frac{1}{C_B}\left(\frac{1}{\sqrt{N}}\sum_jV_{\alpha,\beta}(\bar{X}^i_t-\bar{X}^j_t)-V_{\alpha,\beta}\ast \modif{\bar{\rho}_t}(\bar{X}^i_t)\right)^2\right)\right).
\end{align*}
This way, 
\begin{align*}
B_N\leq& \frac{C_B\lambda^2\|\nabla\bar{\rho}_t\|^2_{L^\infty}}{N^2}\sum_{i}\sum_{\alpha,\beta}\log\int_{\mathds{T}^{dN}}\bar{\rho}_N\exp\left(\frac{1}{C_B}\left(\frac{1}{\sqrt{N}}\sum_jV_{\alpha,\beta}(x_i-x_j)-V_{\alpha,\beta}\ast \modif{\bar{\rho}_t}(x_i)\right)^2\right)\modif{d\textbf{x}}\\
&+C_Bd^2\lambda^2\|\nabla\bar{\rho}_t\|^2_{L^\infty}\mathcal{H}_N(t).
\end{align*}
In the following we choose $C_B=\modif{64e^2}\|V\|_{L^\infty}^2$. Applying Theorem~\ref{borner_A} to 
\begin{align*}
\psi(z,x)=\frac{1}{\modif{8e}\|V\|_{L^\infty}}\left(V(z-x)-V\ast \modif{\bar{\rho}_t}(z)\right),
\end{align*}
which satisfies $\|\psi\|_{L^\infty}\leq\modif{\frac{1}{4e}}$ and is such  that
\begin{align*}
\int_{\mathds{T}^{d}}\psi(z,x)\modif{\bar{\rho}_t}(x)dx=\frac{1}{\modif{8e}\|V\|_{L^\infty}}\int_{\mathds{T}^{d}}V(z-x)\modif{\bar{\rho}_t}(x)dx-\frac{1}{\modif{8e}\|V\|_{L^\infty}}\int_{\mathds{T}^{d}}V\ast \modif{\bar{\rho}_t}(z)\modif{\bar{\rho}_t}(x)dx=0,
\end{align*}
we get 
\begin{equation}\label{maj_B}
B_N\leq \modif{\hat{C}_B}\|V\|_{L^\infty}^2\lambda^2d^2\|\nabla\bar{\rho}_t\|^2_{L^\infty}\left(\mathcal{H}_N(t)+\frac{\tilde{C}_B}{N}\right),
\end{equation}
where \modif{$\hat{C}_B$ and} $\tilde{C}_B$ are universal constants. 

We now proceed with the bound on $A_N$.  Applying Lemma~\ref{changement_de_loi}  to
\begin{align*}
\Phi=\frac{1}{N^2}\sum_{i,j}\left(V(x_i-x_j)-V\ast \modif{\bar{\rho}_t}(x_i)\right):\frac{\nabla_{x_i}^2\bar{\rho}_N}{\bar{\rho}_N},
\end{align*}
we obtain, for all $C_A>0$,
\begin{align*}
A_N\leq& \frac{C_A}{N}\log\int_{\mathds{T}^{dN}}\bar{\rho}_N\exp\left(\frac{1}{C_AN}\sum_{i,j}\left(V(x_i-x_j)-V\ast \modif{\bar{\rho}_t}(x_i)\right):\frac{\nabla_{x_i}^2\bar{\rho}_N}{\bar{\rho}_N}\right)\modif{d\textbf{x}}+C_A\mathcal{H}_N(t)
\end{align*}
In the following we choose
$$C_A=4\sqrt{1600^2+36e^4}\|\nabla^2\bar{\rho}_t\|_{L^\infty}\|V\|_{L^\infty}\lambda\modif{d}:=\hat{C}_A\lambda\modif{d}\|\nabla^2\bar{\rho}_t\|_{L^\infty}\|V\|_{L^\infty}.$$
Then, we apply Theorem~\ref{borner_B} to
\begin{align*}
\phi(z,x)=\frac{1}{C_A}\left(\left(V(z-x)-V\ast \modif{\bar{\rho}_t}(z)\right):\frac{\nabla^2\modif{\bar{\rho}_t}}{\modif{\bar{\rho}_t}}(z)\right),
\end{align*}
which satisfies, thanks to Assumption~\ref{assumption_K}
\begin{align*}
\int_{\mathds{T}^{d}}\phi(z,x)\modif{\bar{\rho}_t}(z)dz=&\frac{1}{C_A}\int_{\mathds{T}^{d}}\left(\left(V(z-x)-V\ast\modif{\bar{\rho}_t}(z)\right):\frac{\nabla^2\modif{\bar{\rho}_t}}{\modif{\bar{\rho}_t}}(z)\right)\modif{\bar{\rho}_t}(z)dz\\
=&\frac{1}{C_A}\int_{\mathds{T}^{d}}\left(\text{div}K(z-x)-\text{div}K\ast\modif{\bar{\rho}_t}(z)\right)\modif{\bar{\rho}_t}(z)dz=0,
\end{align*}
and, thanks to $\int_{\mathds{T}^{d}}\left(V(z-x)-V\ast\modif{\bar{\rho}_t}(z)\right)\modif{\bar{\rho}_t}(x)dx=0$,
\begin{align*}
\int_{\mathds{T}^{d}}\phi(z,x)\modif{\bar{\rho}_t}(x)dx=0.
\end{align*}
Through our choice of $C_A$,  \eqref{gamma} is verified, as $\gamma\leq\left(1600^2+36e^4\right)\left(\frac{2\modif{d}\|V\|_{L^{\infty}}\|\nabla^2\bar{\rho}_t\|_{L^\infty}\lambda}{C_A}\right)^2=\frac{1}{4}<1$. Hence
\begin{equation}\label{maj_A}
A_N\leq \hat{C}_A\|\nabla^2\bar{\rho}_t\|_{L^\infty}\|V\|_{L^\infty}\lambda\modif{d}\left(\mathcal{H}_N(t)+ \frac{\tilde{C}_A}{N}\right),
\end{equation}
where $\hat{C}_A$ and $\tilde{C}_A$ are universal constants. The conclusion easily follows.

\end{proof}

%
%

\subsection{Proof of Theorem~\ref{thm_prop} in the smooth case}

It only remains to gather the previous results. Equations \eqref{maj_H}, \eqref{maj_B} and \eqref{maj_A} yield
\begin{align*}
\frac{d}{dt}\mathcal{H}_N(t)\leq&\left(\hat{C}_A\lambda\modif{d}\|\nabla^2\bar{\rho}_t\|_{L^\infty}\|V\|_{L^\infty}+\frac{\modif{\hat{C}_B}\|V\|_{L^\infty}^2\lambda^2\|\nabla\bar{\rho}_t\|^2_{L^\infty}d^2}{2}\right)\mathcal{H}_N(t)\\
&+\frac{C_2}{N}-\frac{1}{2}\mathcal{I}_N(t),
\end{align*}
and using Corollary~\ref{coro_logsob} and $\hat{C}_A\|\nabla^2\bar{\rho}_t\|_{L^\infty}\|V\|_{L^\infty}\lambda\modif{d}\leq\frac{1}{2}\left(\frac{2\pi}{\lambda}\right)^2+\frac{1}{2}\left(\frac{\lambda}{2\pi}\right)^2\hat{C}^2_A\|\nabla^2\bar{\rho}_t\|^2_{L^\infty}\|V\|^2_{L^\infty}\lambda^2\modif{d^2}$
\begin{align*}
\frac{d}{dt}\mathcal{H}_N(t)\leq&-\left(\left(\frac{2\pi}{\lambda}\right)^2-\hat{C}_A\modif{\lambda d}\|\nabla^2\bar{\rho}_t\|_{L^\infty}\|V\|_{L^\infty}-\frac{\modif{\hat{C}_B}\|V\|_{L^\infty}^2\lambda^2\|\nabla\bar{\rho}_t\|^2_{L^\infty}d^2}{2}\right)\mathcal{H}_N(t)+\frac{C_2}{N}\\
\leq&-\frac{1}{2}\left(\left(\frac{2\pi}{\lambda}\right)^2-\hat{C}_A^2\frac{\lambda^4}{4\pi^2}\modif{ d^2}\|\nabla^2\bar{\rho}_t\|^2_{L^\infty}\|V\|_{L^\infty}^2-\modif{\hat{C}_B}\|V\|_{L^\infty}^2\|\nabla\bar{\rho}_t\|^2_{L^\infty}\lambda^2d^2\right)\mathcal{H}_N(t)+\frac{C_2}{N}.
\end{align*}
In a more concise way, using Lemma~\ref{Bornes_derivees}, it means there are constants $C_1,C_2^\infty,C_3>0$ and a function $t\mapsto C_2(t)>0$ with $\int_0^t C_2(s)ds \leq C_2^\infty$ for all $t\geq0$ such that for all $t\geq0$
\begin{align*}
\frac{d}{dt}\mathcal{H}_N(t)\leq-(C_1-C_2(t))\mathcal{H}_N(t)+\frac{C_3}{N}.
\end{align*}
Multiplying by $\exp(C_1t-\int_0^t C_2(s) ds)$ and integrating in time we get
\begin{align*}
\mathcal H_N(t) &\leq e^{-C_1 t + \int_0^t C_2(s) ds} \mathcal H_N(0) + \frac{C_3}N\int_0^t e^{C_1(s- t) + \int_s^t C_2(u) du}ds \\
&\leq e^{C_2^\infty-C_1 t}\mathcal H_N(\modif{0}) + \frac{C_3}{C_1N}e^{C_2^\infty}, 
\end{align*}
which concludes.

%
%

\subsection{Dealing with the regularity of the density of the particle system}\label{sec:nonsmoothrhon}

As mentioned at the beginning of Section~\ref{Section_Proof}, up to now we have proven the result under the additional assumption that there exists a smooth solution $\rho_N$ to \eqref{part_liou}. Let us now remove this assumption. Consider $(\zeta_\epsilon)_{\epsilon\geq0}$  a sequence of mollifiers such that $\|\zeta_\epsilon\|_{L^1}=1$, whose compact support are assumed to be strictly contained within $\left[-\frac{1}{2},\frac{1}{2}\right]^{d}$. Let us consider $K^\epsilon=K\ast \zeta_\epsilon$. We have $K^\epsilon\in\mathcal{C}^{\infty}(\mathds{T}^d)$ and $\text{div}(K^\epsilon)=0$.

Let $\rho_N^{\epsilon}$ be the unique \modif{smooth} solution \modif{(see Lemma 8 below)} of the parabolic equation with smooth coefficients
\begin{equation}\label{FK_eps}
\partial_t\rho_N^{\epsilon}+\frac{1}{N}\sum_{i,j=1}^NK^\epsilon(x_i-x_j)\cdot\nabla_{x_i}\rho_N^{\epsilon}=\sum_{i=1}^N\Delta_{x_i}\rho_N^{\epsilon},
\end{equation}
with initial condition $\rho_N^{\epsilon}(0,\cdot)=\rho_N(0,\cdot).$. 

We have the following bounds
\begin{lemma}\label{bound_part}
Let $\gamma>1$ be such that $\rho_N(0)\in\mathcal C_\gamma^\infty(\mathds{T}^{dN})$. Then, for all $t\geq 0$ and all $\epsilon>0$, $\rho_N^\epsilon(t)\in\mathcal C_\gamma^\infty(\mathds{T}^{dN})$.
\end{lemma}
\begin{proof}
Let $\textbf{x}\in\mathds{T}^{dN}$ Consider the particle system $dX_i^\epsilon(t)=-\frac{1}{N}\sum_{j=1}^NK^\epsilon(X_i^\epsilon(t)-X_j^\epsilon(t))dt+\sqrt{2}dB^i_t$ with initial condition $\textbf{X}^\epsilon_0=\textbf{x}$, where we denote $\textbf{X}^\epsilon_t=(X_1^\epsilon(t),\dots,X_N^\epsilon(t))$. We have strong existence and uniqueness for this SDE. Then
\begin{align*}
\rho_N^{\epsilon}(t,\textbf{x})=\mathbb{E}\left(\rho_N^{\epsilon}(0,\textbf{X}^\epsilon_t)\right).
\end{align*}
The bounds on $\rho_N^{\epsilon}$ follow.
\end{proof}

Using Lemma~\ref{bound_part}, we get $(\rho_N^{\epsilon})_\epsilon$ is a sequence of smooth functions uniformly bounded in $L^{\infty}(\mathbb{R}^+\times\mathds{T}^{Nd})$. This yields two results.

First, we can extract a weakly-* converging subsequence in $L^{\infty}(\mathbb{R}^+\times\mathds{T}^{Nd})$, i.e there exists ${\rho_N\in L^{\infty}(\mathbb{R}^+\times\mathds{T}^{Nd})}$ such that for all $f\in L^{1}(\mathbb{R}^+\times\mathds{T}^{Nd})$ we have $$\int_{\mathds{T}^{Nd}}\rho_N^{\epsilon}f\underset{\epsilon\rightarrow 0^+} {\longrightarrow}\int_{\mathds{T}^{Nd}}\rho_Nf.$$
We finally check that $\rho_N$ is indeed a weak solution of \eqref{part_liou}.  For all $T\geq0$ and for all $f$ smooth test function on $[0,T]\times\mathds{T}^{Nd}$
\begin{description}
\item[$\bullet$] We have, since $\partial_tf$ \modif{is} smooth and therefore in $L^1([0,T]\times\mathds{T}^{Nd})$
\begin{align*}
\int_{\mathds{T}^{Nd}} \rho_N^{\epsilon} \partial_tf\rightarrow\int_{\mathds{T}^{Nd}} \rho_N\partial_tf.
\end{align*}
\item[$\bullet$] Likewise, since $\Delta_{x_i}f$ \modif{is} smooth and therefore in $L^1([0,T]\times\mathds{T}^{Nd})$
\begin{align*}
\int_{\mathds{T}^{Nd}} \rho_N^{\epsilon}\Delta_{x_i}f\rightarrow\int_{\mathds{T}^{Nd}} \rho_N\Delta_{x_i}f.
\end{align*}
\item[$\bullet$] Finally
\begin{align*}
\int_{\mathds{T}^{Nd}}\rho_N^{\epsilon}K^\epsilon(x_i-x_j)&\cdot\nabla_{x_i}f-\int_{\mathds{T}^{Nd}}\rho_NK(x_i-x_j)\cdot\nabla_{x_i}f\\
=&\int_{\mathds{T}^{Nd}}\rho_N^{\epsilon}\left(K^\epsilon(x_i-x_j)-K(x_i-x_j)\right)\cdot\nabla_{x_i}f+\int_{\mathds{T}^{Nd}}\left(\rho_N^{\epsilon}-\rho_N\right)K(x_i-x_j)\cdot\nabla_{x_i}f\\
\leq& \|\rho_N^{\epsilon}\|_{L^\infty}\|\nabla_{x_i}f\|_{L^\infty}\|K^\epsilon-K\|_{L^1}+\int_{\mathds{T}^{Nd}}\left(\rho_N^{\epsilon}-\rho_N\right)K(x_i-x_j)\cdot\nabla_{x_i}f\\
\rightarrow&0,
\end{align*}
as $\|K^\epsilon-K\|_{L^1}\rightarrow0$ and $K(x_i-x_j)\cdot\nabla_{x_i}f\in L^{1}([0,T]\times\mathds{T}^{Nd})$.
\end{description}
We have thus proven that $\rho_N$ is a weak solution of \eqref{part_liou}. 

Likewise, we may consider $(\bar{\rho}^{\epsilon})_\epsilon$, which weakly-* converges to a solution which, by uniqueness, is $\bar{\rho}$.

Second, $\rho_N^{\epsilon}$ satisfies the assumption made at the beginning of Section~\ref{Section_Proof}, i.e $\rho_N^{\epsilon}\in\mathcal{C}^\infty_{>0}(\mathbb{R}^+\times\mathds{T}^d)$. Since by considering $V^\epsilon=V\ast\zeta_\epsilon$  we have $K^\epsilon=\text{div}\left(V^\epsilon\right)$, we get that $K^\epsilon$ satisfies Assumption~\ref{assumption_K} and that the calculations done in Section~\ref{Section_Proof} are valid for this specific kernel, i.e 
\begin{equation}\label{res_reg}
\mathcal{H}_N(\rho^\epsilon_N\modif{(t)},\bar{\rho}^\epsilon_N\modif{(t)})\leq\mathcal{H}_N(\rho_N\modif{(0)},\bar{\rho}_N\modif{(0)})e^{-C_1^\epsilon t}e^{C^{\infty,\epsilon}}+\frac{C^\epsilon_3e^{C^{\infty,\epsilon}}}{C^\epsilon_1}\frac{1}{N}.
\end{equation}

Notice how, in the proof of Lemma~\ref{Bornes_derivees}, the constants bounding the various derivatives of $\bar{\rho}$ only depend on the initial conditions, on $\|K\|_{L^1}$ and on $\|V\|_{L^\infty}$. Since $(\zeta_\epsilon)_{\epsilon\geq0}$ is a sequence of mollifiers, we have ${\|K^\epsilon\|_{L^1}\rightarrow\|K\|_{L^1}}$ as $\epsilon$ tends to 0, and $\|V^\epsilon\|_{L^\infty}\leq\|V\|_{L^\infty}$. 
The righthand side of \eqref{res_reg} can thus be chosen independent of $\epsilon$.

We now use the fact that for $u\geq0$ and $v\in\mathbb{R}$ we have $ uv\leq u\log u-u+e^v,$ to obtain the variational formulation of the entropy,
\begin{equation}\label{var_ent}
N\mathcal{H}_N(\rho^\epsilon_N\modif{(t)},\bar{\rho}^\epsilon_N\modif{(t)})=\sup\left\{\mathbb{E}_{\rho^\epsilon_N\modif{(t)}}(g)-\mathbb{E}_{\bar{\rho}^\epsilon_N\modif{(t)}}\left(e^g\right)+1,g\in L^\infty\right\},
\end{equation}
the equality being attained for $g=\log\left(\frac{\rho^\epsilon_N}{\bar{\rho}^\epsilon_N}\right)$. We thus consider, for $g\in L^\infty$,
\begin{equation*}
\frac{1}{N}\left(\mathbb{E}_{\rho^\epsilon_N\modif{(t)}}(g)-\mathbb{E}_{\bar{\rho}^\epsilon_N\modif{(t)}}\left(e^{g}\right)+1\right)\leq\mathcal{H}_N(\rho_N\modif{(0)},\bar{\rho}_N\modif{(0)})e^{-C_1 t}e^{C^{\infty}}+\frac{C_3e^{C^{\infty}}}{1+C_1}\frac{1}{N}.
\end{equation*}
By definition of the weak-* convergence in $L^\infty$ (since both $g$ and $e^g$ are thus in $L^1$), we have the following convergence  ${\mathbb{E}_{\rho^\epsilon_N\modif{(t)}}(g)\longrightarrow\mathbb{E}_{\rho_N\modif{(t)}}(g)}$ and ${\mathbb{E}_{\bar{\rho}^\epsilon_N\modif{(t)}}\left(e^{g}\right)\longrightarrow\mathbb{E}_{\bar{\rho}_N\modif{(t)}}\left(e^{g}\right)}$ as $\epsilon$ tends to 0.
Therefore, for all $g\in L^\infty$,
\begin{equation*}
\frac{1}{N}\left(\mathbb{E}_{\rho_N\modif{(t)}}(g)-\mathbb{E}_{\bar{\rho}_N\modif{(t)}}\left(e^{g}\right)+1\right)\leq\mathcal{H}_N(\rho_N\modif{(0)},\bar{\rho}_N\modif{(0)})e^{-C_1 t}e^{C^{\infty}}+\frac{C_3e^{C^{\infty}}}{1+C_1}\frac{1}{N},
\end{equation*}
which yields Theorem~\ref{thm_prop}, using \eqref{var_ent} for $\mathcal{H}_N(\rho_N\modif{(t)},\bar{\rho}_N\modif{(t)})$.

%
%

\subsection{Proof of Corollary~\ref{coro_wass} }

Let $k\in\mathbb{N}$, and $N\geq k$. The sub-additivity of the entropy (see for instance Theorem 10.2.3 of \cite{ABC+00}) implies that the (rescaled) relative entropy of the marginals is bounded by the total (rescaled) relative entropy 
\begin{align*}
k\left\lfloor\frac{N}{k}\right\rfloor\mathcal{H}_k(\rho_{N}^{k}\modif{(t)},\bar{\rho}_k\modif{(t)})\leq N\mathcal{H}_N(\rho_N\modif{(t)},\bar{\rho}_N\modif{(t)}).
\end{align*}
The logarithmic Sobolev inequality established in Corollary~\ref{coro_logsob} implies a Talagrand's transportation inequality (see \cite{OV00}), so that the $L^2$-Wasserstein distance is bounded by the relative entropy. Classically, this is also the case of the total variation thanks to Pinsker's inequality, and thus
\begin{align*}
\|\rho_N^k(t) -\bar \rho_k(t)\|_{L^1}+\mathcal{W}_2(\rho_{N}^{k}\modif{(t)},\bar{\rho}_k\modif{(t)})\leq C\sqrt{k\mathcal{H}_k(\rho_{N}^{k}\modif{(t)},\bar{\rho}_k\modif{(t)})}\leq C\sqrt{\frac{N}{\left\lfloor\frac{N}{k}\right\rfloor}\mathcal{H}_N(\rho_{N}\modif{(t)},\bar{\rho}_N\modif{(t)})}.
\end{align*}
With the additional assumption that $\mathcal{H}_N(\rho_{N}\modif{(0)},\bar{\rho}_N\modif{(0)})=0$, we thus get the result using Theorem~\ref{thm_prop}. To obtain the result on the empirical measure, we recall for the sake of completeness the arguments of \cite[Proposition 8]{JM22}. Given $x,y\in\mathds{T}^{dN}$, a coupling of $\pi(x)$ and $\pi(y)$ is obtained by considering $(x_J,y_J)$ where $J$ is uniformly distributed over $\llbracket 1,N\rrbracket$. From this we get $\mathcal W_2(\pi(x),\pi(y))\leqslant |x-y|/\sqrt N$. Considering $(\mathbf{X},\mathbf{Y})$ an optimal coupling of $(\rho_N(t),\bar\rho_N(t))$, we bound
\begin{align*}
\mathbb E\left(\mathcal W_2(\pi(\mathbf{X},\bar\rho_{\modif{t}})\right) & \leqslant \mathbb E\left(\mathcal W_2(\pi(\mathbf{X}),\pi(\mathbf{Y}))\right)  + \mathbb E\left(\mathcal W_2(\pi(\mathbf{Y}),\bar\rho_{\modif{t}})\right) \\
& \leqslant \frac{1}{\sqrt N}  \mathcal W_2(\rho_N(t),\bar\rho_N(t)) + \mathbb E\left(\mathcal W_2(\pi(\mathbf{Y}),\bar\rho_{\modif{t}})\right) . 
\end{align*}
The last term is tackled with the result for i.i.d. variables established in \cite{FG15}.

%
%
%
%

\appendix

%
%

\section{Proof of Theorem~\ref{exist}}\label{Appendice_reg}

The proof is based on an iterative procedure, and relies heavily on the work of Ben-Artzi \cite{Ben94}. Let $\bar{\rho}^{(-1)}:=0$, and then for $k\in\mathbb{N}$ solve
\begin{align}
\partial_t\bar{\rho}^{(k)}=&-\left(u^{(k-1)}\cdot\nabla\right)\bar{\rho}^{(k)}+\Delta\bar{\rho}^{(k)}, \text{ in } \mathbb{R}^+\times\mathds{T}^d \label{it_1} \\
u^{(k)}=&K\ast\bar{\rho}^{(k)}\\
\bar{\rho}^{(k)}(0,\cdot)=&\mu_0 \label{it_3}.
\end{align}
Let us recall the following lemma concerning the regularity of the second order parabolic equation. We refer to Chapter 7 of \cite{Eva10} for a proof on a bounded domain that can be extended to the torus.
\begin{lemma}\label{lemma_heat}
Let $a(t,x)$ be a $\mathcal{C}^\infty$ function on $\mathbb{R}^+\times\mathds{T}^d$ and $\psi_0\in\mathcal C^\infty(\mathds{T}^d)$. Then the problem
\begin{equation*}
\begin{array}{ll}
        \partial_t\psi=- a\cdot\nabla\psi+\Delta\psi, \text{ in } \mathbb{R}^+\times\mathds{T}^d \\
        \psi(0,\cdot)=\psi_0,
    \end{array}
\end{equation*}
has a unique solution, which is $\mathcal{C}^\infty$.
\end{lemma}

\begin{lemma}
Suppose $\mu_0\in\mathcal{C}^\infty(\mathds{T}^d)$. Then the system \eqref{it_1}-\eqref{it_3} defines successively a sequence of $\mathcal{C}^\infty$ solutions $\{\bar{\rho}^{(k)},u^{(k)}\}_{k\in\mathbb{N}}$.\\ Furthermore, for all $t\geq0$ and all $k\in\mathbb{N}$, $\|\bar{\rho}^{(k)}(t,\cdot)\|_{L^\infty}\leq \|\mu_0\|_{L^\infty}$ and $\|u^{(k)}(t,\cdot)\|_{L^\infty}\leq \|K\|_{L^1}\|\mu_0\|_{L^\infty}$. \\
Finally, given a final time $T\geq0$, $\bar{\rho}^{(k)}$ (resp. $u^{(k)}$) and all its derivatives, both in time and in space, are bounded on $[0,T]\times \mathds{T}^d$ uniformly in $k$
\end{lemma}

\begin{proof}
We use induction on $k$. The assertion is clear for $\bar{\rho}^{(0)}$ as the explicit solution to the heat equation. Suppose $\{\bar{\rho}^{(j)},u^{(j-1)}\}_{j=0,...,k}$ have be shown to be $\mathcal{C}^\infty$ solutions bounded uniformly in time.
\paragraph{Regularity.}  By definition
\begin{align*}
u^{(k)}(t,x)=K\ast\bar{\rho}^{(k)}(t,x)=\int_{\mathds{T}^d}K(x-y)\bar{\rho}^{(k)}(t,y)dy=-\int_{\mathds{T}^d}K(y)\bar{\rho}^{(k)}(t,x-y)dy.
\end{align*}
Then
\begin{align*}
u^{(k)}(t,x)=-\int_{\mathds{T}^d}\text{div}V(y)\bar{\rho}^{(k)}(t,x-y)dy=-\int_{\mathds{T}^d}V(y)\nabla_y\bar{\rho}^{(k)}(t,x-y)dy.
\end{align*}
Since we are in the compact set $\mathds{T}^d$, that $V\in L^{\infty}(\mathds{T}^d)$, and that $\bar{\rho}^{(k)}\in\mathcal{C}^\infty(\mathbb{R}^+\times\mathds{T}^d)$ by induction hypothesis, we can easily show that $u^{(k)}$, as well as all its derivatives, are Lipschitz continuous. Hence $u^{(k)}$ is $\mathcal{C}^\infty$. Using Lemma~\ref{lemma_heat} in \eqref{it_1} with $k$ replaced by $k+1$ yields the desired result for $\bar{\rho}^{(k+1)}$. 

\paragraph{Boundedness of $\bar{\rho}^{(k+1)}$ and $u^{(k)}$.} Let us show that for all $T\geq0$, $\bar{\rho}^{(k+1)}$ and $u^{(k)}$ are both bounded on $[0,T]\times\mathds{T}^d$, with a bound independent of $T$. We have, using Young's convolution inequality and the induction hypothesis
\begin{align*}
\|u^{(k)}(t,\cdot)\|_{L^\infty}\leq \|K\|_{L^1}\|\bar{\rho}^{(k)}(t,\cdot)\|_{L^\infty}\leq\|K\|_{L^1}\|\mu_0\|_{L^\infty}.
\end{align*}
Now $\bar{\rho}^{(k+1)}$ is the unique solution of
\begin{align*}
\partial_t\bar{\rho}^{(k+1)}=&-\left(u^{(k)}\cdot\nabla\right)\bar{\rho}^{(k+1)}+\Delta\bar{\rho}^{(k+1)}\\
\bar{\rho}^{(k+1)}(0,x)=&\mu_0(x).
\end{align*}
For $t\geqslant 0$, consider $Z^{(k+1)}_s$ the  strong solution of the following stochastic differential equation for $s\in[0,t]$
\begin{equation*}
 dZ^{(k+1)}_s=\sqrt{2}dB_s-u^{(k)}(t-s,Z_s)ds,
\end{equation*}
which exists, is unique  and non-explosive since $u^{(k)}$ is smooth, bounded and Lipschitz continuous. Then 
\[\bar{\rho}^{(k+1)}(t,x) = \mathbb E_x\left( \mu_0(Z^{(k+1)}_t)\right).\]
We thus get  $$\|\bar{\rho}^{(k+1)}_t\|_{L^\infty}\leq \|\mu_0\|_{L^\infty}.$$
\modif{Notice that this is simply a probabilistic way of presenting the use of the maximum principle.} 

\paragraph{Boundedness of the derivatives of $\bar{\rho}^{(k+1)}$ and $u^{(k)}$.} The boundedness of the derivatives of $u^{(k)}$ is a direct consequence of the boundedness of the derivatives of $\bar{\rho}^{(k)}$ thanks to Young's convolution inequality. Then, the proof for $\bar{\rho}^{(k+1)}$  similar to the proof of Lemma~\ref{Bornes_derivees}, using the boundedness of the derivatives of $u^{(k)}$. To show that the bounds are in fact independent of $k$, we follow the proof of Lemma~\ref{Bornes_derivees}, i.e by induction on the order of the derivative, and within each induction step we prove that both the integrated and uniform bounds are independent of $k$. This comes from the fact that the proof initially only relies on the bounds on $\|\bar{\rho}^{(k+1)}_t\|_{L^\infty}$ and $\|u^{(k)}_t\|_{L^\infty}$ -which, as we have shown, only depend on $\|\mu_0\|_{L^\infty}$- and then for each induction step on the initial condition and on the bounds constructed at the previous step (therefore independent of $k$). The bounds concerning the derivatives involving time are then obtained thanks to the bounds on the space derivatives using \eqref{it_1}.

\end{proof}

\textit{Proof of Theorem~\ref{exist}.}
It is sufficient to prove existence and uniqueness of the solution in $[0,T]\times\mathds{T}^d$ for all $T\geq0$, since then the solutions on $[0,T_1]\times\mathds{T}^d$ and $[0,T_2]\times\mathds{T}^d$, with $T_1<T_2$, must coincide in $[0,T_1]\times\mathds{T}^d$, leading to the existence and uniqueness of the global solution in $\mathbb{R}^+\times\mathds{T}^d$. Let us consider $T\geq0$

\paragraph{Existence in $[0,T]\times\mathds{T}^d$ for $T$ small enough:} Let us show the existence of the limit solution. We consider here $T$ to be small enough (an explicit bound will be given later)
Let $G(t,x)=\sum_{k\in\mathbb{Z}^d}\frac{1}{(4\pi t)^\frac{d}{2}}\exp(-\frac{|x+k|^2}{4 t})$ be the heat kernel on the d dimensional torus. We have
\begin{align*}
\bar{\rho}^{(k)}(t,x)=G(t,\cdot)\ast\mu_0(x)-\int_0^t\int_{\mathds{T}^d}G(t-s,x-y)u^{(k-1)}(s,y)\cdot\nabla_y \bar{\rho}^{(k)}(s,y)dyds.
\end{align*}
Let us denote $N_k(t)=\sup_{0\leq s\leq t}\|\bar{\rho}^{(k+1)}(s,\cdot)-\bar{\rho}^{(k)}(s,\cdot)\|_{L^\infty}$. We have, using $\nabla_y\cdot u^{(k)}=0$
\begin{align*}
\bar{\rho}^{(k+1)}(t,x)-\bar{\rho}^{(k)}(t,x)=&-\int_0^t\int_{\mathds{T}^d}\nabla_yG(t-s,x-y)\left(\bar{\rho}^{(k+1)}(s,y)-\bar{\rho}^{(k)}(s,y)\right)u^{(k
)}(s,y)dyds\\
&-\int_0^t\int_{\mathds{T}^d}\nabla_yG(t-s,x-y)\bar{\rho}^{(k)}(s,y)\left(u^{(k)}(s,y)-u^{(k-1)}(s,y)\right)dyds.
\end{align*}
Remark (using the first moment of the chi distribution) that, for some constant $\beta>0$
\begin{align*}
\int_{\mathds{T}^d}|\nabla_xG(t,x)|dx\leq\beta t^{-\frac{1}{2}}.
\end{align*}
We thus get
\begin{align*}
\|\bar{\rho}^{(k+1)}(t,\cdot)-\bar{\rho}^{(k)}(t,\cdot)\|_{L^\infty}\leq&\beta\|K\|_{L^1}\|\mu_0\|_{L^\infty}\int_0^t(t-s)^{-\frac{1}{2}}\|\bar{\rho}^{(k+1)}(s,\cdot)-\bar{\rho}^{(k)}(s,\cdot)\|_{L^\infty}ds\\
&+\beta\|\mu_0\|_{L^\infty}\int_0^t(t-s)^{-\frac{1}{2}}\|u^{(k)}(s,\cdot)-u^{(k-1)}(s,\cdot)\|_{L^\infty}ds,
\end{align*}
and
\begin{align*}
\|u^{(k)}(s,\cdot)-u^{(k-1)}(s,\cdot)\|_{L^\infty}\leq \|K\|_{L^1}\|\bar{\rho}^{(k)}(s,\cdot)-\bar{\rho}^{(k-1)}(s,\cdot)\|_{L^\infty}.
\end{align*}
Therefore
\begin{align*}
N_k(t)\leq\beta\|K\|_{L^1}\|\mu_0\|_{L^\infty}\int_0^t(t-s)^{-\frac{1}{2}}N_k(s)ds+\beta\|K\|_{L^1}\|\mu_0\|_{L^\infty}\int_0^t(t-s)^{-\frac{1}{2}}N_{k-1}(s)ds.
\end{align*}
Denoting $C=\beta\|K\|_{L^1}\|\mu_0\|_{L^\infty}$ we get 
\begin{equation}\label{ineq_rec_Q_k}
N_k(t)\leq C\int_0^t(t-s)^{-\frac{1}{2}}\left(N_k(s)+N_{k-1}(s)\right)ds.
\end{equation}
Since $N_k$ is continuous, there exists $R>0$ such that for all $t\in[0,T]$ we have $N_k(t)\leq R$. We thus have, using this bound in \eqref{ineq_rec_Q_k} and assuming $2C\sqrt{T}\leq\frac{1}{2}$
\begin{align*}
N_k(t)\leq& RC\int_0^t(t-s)^{-\frac{1}{2}}ds+C\int_0^t(t-s)^{-\frac{1}{2}}N_{k-1}(s)ds\\
\leq&\frac{R}{2}+C\int_0^t(t-s)^{-\frac{1}{2}}N_{k-1}(s)ds.
\end{align*}
We use this bound in \eqref{ineq_rec_Q_k}
\begin{align*}
N_k(t)\leq& \frac{R}{2}C\int_0^t(t-s)^{-\frac{1}{2}}ds +C\int_0^t(t-s)^{-\frac{1}{2}}N_{k-1}(s)ds\\
&+C^2\int_0^t\int_0^s(t-s)^{-\frac{1}{2}}(s-u)^{-\frac{1}{2}}N_{k-1}(u)duds.
\end{align*}
We deal with this last term
\begin{align*}
C^2\int_0^t\int_0^s&(t-s)^{-1/2}(s-u)^{-\frac{1}{2}}N_{k-1}(u)duds\\
=&C^2\int_0^tN_{k-1}(u)\int_u^t(t-s)^{-\frac{1}{2}}(s-u)^{-\frac{1}{2}}dsdu\\
=&C^2\pi\int_0^tN_{k-1}(u)du.
\end{align*}
Let $\alpha=\sqrt{T}\pi C$ and choose $T$ such that $\alpha\leq\frac{1}{2}$ (which in turns also yields the previous condition ${2C\sqrt{T}\leq\frac{1}{2}}$). We have 
\begin{align*}
\alpha C\int_0^t(t-s)^{-\frac{1}{2}}N_{k-1}(s)ds-C^2\pi\int_0^tN_{k-1}(s)du=C\int_0^tN_{k-1}(s)\left(\alpha(t-s)^{-\frac{1}{2}}-\pi C\right)ds,
\end{align*}
and since $\alpha=\sqrt{T}\pi C\geq\sqrt{t-s}\pi C$ for $0\leq s\leq t\leq T$, we get
\begin{align*}
\alpha C\int_0^t(t-s)^{-\frac{1}{2}}N_{k-1}(s)ds\geq C^2\pi\int_0^tN_{k-1}(s)du,
\end{align*}
and thus
\begin{align*}
N_k(t)\leq& \frac{R}{4} +C(1+\alpha)\int_0^t(t-s)^{-\frac{1}{2}}N_{k-1}(s)ds.
\end{align*}
Iterating this method, we obtain for all $n\in\mathbb{N}$
\begin{align*}
N_k(t)\leq& 2^{-n}R +C(1+\alpha+\dots+\alpha^{n-1})\int_0^t(t-s)^{-\frac{1}{2}}N_{k-1}(s)ds,
\end{align*}
and thus
\begin{align*}
N_k(t)\leq2C\int_0^t(t-s)^{-\frac{1}{2}}N_{k-1}(s)ds.
\end{align*}
We now show that this implies that
\begin{equation}\label{relation_rec_Q}
N_k(t)\leq N_0(T)\left(2C\Gamma\left(\frac{1}{2}\right)\right)^k t^{k/2}\Gamma\left(\frac{k+2}{2}\right)^{-1},
\end{equation}
where we denote $\Gamma\left(z\right)=\int_0^\infty t^{z-1}e^{-t}dt$. We have that for $k=0$, \eqref{relation_rec_Q} is satisfied and, by induction, we have
\begin{align*}
\int_0^t(t-s)^{-\frac{1}{2}}s^{\frac{k}{2}}ds=t^{\frac{k+1}{2}}\int_0^1(1-u)^{-\frac{1}{2}}u^{\frac{k}{2}}du=t^{\frac{k+1}{2}}\frac{\Gamma\left(\frac{1}{2}\right)\Gamma\left(\frac{k+2}{2}\right)}{\Gamma\left(\frac{k+3}{2}\right)}
\end{align*}
Using the fact that $\Gamma(k+1)=k!$ and $\Gamma(k+\frac{3}{2})=k!\Gamma(\frac{1}{2})$, we get that $\sum_{k=0}^\infty N_k(t)$ converges uniformly for $t\in[0,T]$ and the limits
\begin{align*}
\bar{\rho}(t,x)=\lim_{k\rightarrow\infty}\bar{\rho}^{(k)}(t,x)\ \text{ and }\ u(t,x)=\lim_{k\rightarrow\infty}u^{(k)}(t,x)
\end{align*}
exist in $\mathcal{C}([0,T]\times\mathds{T}^d)$. Now, since for all $l,n\in\mathbb{N}$ and all $\alpha_1,...,\alpha_n$, $\|\partial^l_{t}\partial_{\alpha_1,...,\alpha_n}\bar{\rho}^{(k)}\|_{L^\infty}$ and $\|\partial^l_{t}\partial_{\alpha_1,...,\alpha_n}u^{(k)}\|_{L^\infty}$ are bounded uniformly in $k$, using Arzela-Ascoli theorem, we have uniform convergence, up to an extraction, of the derivatives. Hence the validity of the limits in $\mathcal{C}^{\infty}([0,T]\times\mathds{T}^d)$, i.e there is convergence of the functions along with their derivatives of all order in $[0,T]\times\mathds{T}^d$. This gives us the fact that the limit $\bar{\rho}$ satisfies \eqref{PDE_vortex}.

\paragraph{Uniqueness in $[0,T]\times\mathds{T}^d$.} Suppose $\bar{\rho}^1$ and $\bar{\rho}^2$ are two bounded solutions of \eqref{PDE_vortex} on $[0,T]\times\mathds{T}^d$. Then
\begin{align*}
\partial_t\left(\bar{\rho}^1-\bar{\rho}^2\right)-\Delta\left(\bar{\rho}^1-\bar{\rho}^2\right)=-\left(K\ast\bar{\rho}^1\right)\cdot\nabla\left(\bar{\rho}^1-\bar{\rho}^2\right)-\nabla\cdot\left(\left(K\ast\bar{\rho}^1-K\ast\bar{\rho}^2\right)\bar{\rho}^2\right),
\end{align*}
so that
\begin{align*}
\bar{\rho}^1(t,x)-\bar{\rho}^2(t,x)=&-\int_0^t\int_{\mathds{T}^d}\nabla_yG(t-s,x-y)\cdot\left(K\ast_y\bar{\rho}^1(s,y)\right)\left(\bar{\rho}^1(s,y)-\bar{\rho}^2(s,y)\right)dyds\\
&-\int_0^t\int_{\mathds{T}^d}\nabla_yG(x-y,t-s)\cdot\left(K\ast_y\bar{\rho}^1(s,y)-K\ast_y\bar{\rho}^2(s,y)\right)\bar{\rho}^2(s,y)dyds.
\end{align*}
Let $N(t):=\sup_{0\leq s\leq t}\|\bar{\rho}^1(s,\cdot)-\bar{\rho}^2(s,\cdot)\|_{L^\infty}.$ Recall
\begin{align*}
\|K\ast\bar{\rho}^1(s,\cdot)-K\ast\bar{\rho}^2(s,\cdot)\|_{L^\infty}\leq \|K\|_{L^1}\|\bar{\rho}^1(s,\cdot)-\bar{\rho}^2(s,\cdot)\|_{L^\infty},
\end{align*}
which implies, like previously, the existence of a constant $C$ such that
\begin{align*}
N(t)\leq C\int_0^t(t-s)^{-1/2}N(s)ds.
\end{align*}
We choose $L>0$ such that $C\int_0^Ts^{-\frac{1}{2}}e^{-Ls}ds\leq\frac{1}{2}$, and let $Q(t)=e^{-Lt}N(t)$, which satisfies for all $t$
\begin{align*}
Q(t)\leq C\int_0^t(t-s)^{-1/2}Q(s) e^{-L(t-s)}ds.
\end{align*}
Let $R>0$ be such that $Q(t) \leq R$.

Then 
\begin{align*}
Q(t)\leq RC\int_0^t(t-s)^{-1/2} e^{-L(t-s)}ds\leq\frac{R}{2}.
\end{align*}
By induction, we get $N(t)=0$ for $t\in[0,T]$. This concludes the proof of uniqueness.

\paragraph{Existence in $\mathbb{R}^+\times\mathds{T}^d$.} For $T$ small enough, there exists a solution in $[0,T]\times\mathds{T}^d$. Notice that $T$ only depends on constants independent of time (it depends on the $L^\infty$ bound of the initial condition, which we have shown propagates). It is therefore possible to construct the (unique) smooth solution on all intervals $[t_0,T+t_0]\times\mathds{T}^d$. Uniqueness allows us to iteratively construct the (unique) smooth solution on $\mathbb{R}^+\times\mathds{T}^d$. This concludes the proof.
\begin{flushright}$\openbox$\end{flushright}

%
%

{\bf Acknowledgements}\\
This work has been (partially) supported by the Project EFI ANR-17-CE40-0030 of the French National Research Agency. The authors sincerely thank Didier Bresch, Laurent Chupin and Nicolas Fournier for many discussions around this topic.

\bibliographystyle{alpha}
\bibliography{biblio_2023_JW}

\end{document}